\documentclass[11pt]{amsart} %{article}
\usepackage[utf8]{inputenc}
\usepackage[english]{babel} 
\usepackage{amsthm}
\usepackage{amsmath, mathtools, amssymb, mathrsfs}
\usepackage{adjustbox}
\usepackage[colorlinks=true,linktocpage=true,linkcolor=blue,citecolor=red]{hyperref}
\usepackage[capitalise]{cleveref}
\usepackage{comment}
\usepackage{tikz,tikz-cd}
\usepackage{xypic}
\usepackage{graphicx}
\usepackage{amsfonts}
\usepackage{xcolor}
\usepackage{enumitem}   
\usepackage[margin=3.25cm]{geometry}
\usepackage{aliascnt}
\newtheorem{theorem}{Theorem}[section]
\newtheorem{corollary}[theorem]{Corollary}

\newtheorem{lemma}[theorem]{Lemma}
\theoremstyle{definition}
\newtheorem{definition}[theorem]{Definition}

\theoremstyle{definition}
\newtheorem*{definition*}{Notational Conventions}
\newtheorem*{theorem*}{Theorem}
\theoremstyle{remark}
\newtheorem{remark}[theorem]{Remark}

\theoremstyle{definition}
\newtheorem{example}[theorem]{Example}
\newtheorem*{example*}{Example}
\theoremstyle{definition}
\newtheorem{construction}[theorem]{Construction}
\newtheorem{notation}[theorem]{Notation}

\theoremstyle{definition}

\theoremstyle{definition}

\theoremstyle{definition}

\title{Classifying Infinity Topoi via weighted limits}
\author{Ivan Di Liberti}
\author{Nicholas J. Meadows}

\begin{document}

\begin{abstract}
We construct classifying $\infty$-topoi by showing that the $(\infty,2)$-category of topoi has weighted limits. We show that several prestacks of interest have a classifying topos, including the prestack of spectra.
\end{abstract}

\maketitle
\tableofcontents
\section*{Introduction}

\subsection*{Landscape}

\subsubsection*{The notion of classifying topos}
The theory of \textit{classifying topoi} represents a central aspect of topos theory. Its relevance emerged from the Grothendieck school, who initially identified the \textit{classifying topos} of $G$-torsors for $G$ a group scheme (see the introduction of \cite{moerdijk2006classifying}). The notion was later developed by the work of Monique Hakim in her PhD thesis, mostly motivated by ringed topoi \cite{hakim2019topos}. Finally, a complete theory of classifying topoi emerged in the community of categorical logic via the notion of \textit{geometric theory} \cite{Johnstone,talkj}.

The classifying topos $\mathsf{Set}[\mathbb{T}]$ of a geometric theory $\mathbb{T}$ provides access to the so-called \textit{generic model} of a theory, which controls the generic behavior of the whole theory. 

Besides many applications to logic (see part D of \cite{Johnstone}), some early applications to algebraic geometry were observed by Wraith \cite{wraith2006generic}, and later displayed by Blechschmidt \cite{blechschmidt2021using,blechschmidt2018elementary}, whose work has influenced and inspired many other researchers in the last years \cite{blechschmidt2020general,cherubini2024foundation,cherubini2024projective} in the general topic of synthetic algebraic geometry. Many other applications of this technology to other fields of mathematics have been provided by Caramello \cite{caramello2023unification,caramello2016grothendieck}.

\subsubsection*{Towards classifying $\infty$-topoi}
This paper deals with the general problem of building \textit{classifying} $\infty$-topoi. There are several motivations to approach this problem, some very concrete and some very conceptual. Many of them can be understood as a homotopical replay of the classical motivations, investing several topics of pure mathematics. Others protrude to the question of what \textit{homotopical geometric logic} is supposed to be, and thus pertain more to logic. The question concerning the existence of some specific classifying $\infty$-topoi has appeared a number of times in the community of homotopical algebra \cite{215771,345680} and it was very informally addressed in the last section of \cite{anel2022left}.

In the last years there have been some attempts to define a satisfactory notion of $\infty$-geometric theory. Prominently, Stenzel \cite{stenzel2022higher,stenzel2023notions} has provided the best results so far, introducing notions of \textit{regular}, \textit{coherent} and \textit{geometric} $\infty$-theories. Unfortunately, these notions are not supported by a proof calculus which is as robust as that of geometric logic, and require a much more sophisticated treatment than the usual $1$-dimensional counterpart. The main obstruction to such development at the moment seems to be conceptual. An $\infty$-topos has an object classifier and the complexity of its localizations cannot be reduced to its classifier of subobjects. This is mirrored by the complexity of the theory of sites, which requires the more flexible notion \cite{anel2022left,anel2024left} to be developed, while the more familiar notion of $\infty$-site by Lurie \cite{LurieHTT} only recovers hypercomplete topoi. Currently it is unclear (and to some extent ultimately open to debate) whether a compact notion of $\infty$-geometric theory can be defined without the complexity of Stenzel, and in the quest for defining classifying topoi, the question \textit{classifying topoi of what?} remains somewhat unanswered.

\subsection*{Contribution}

In this paper, we suspend our judgment concerning the notion of geometric $\infty$-theory and take a more concrete approach to build classifying topoi. In \cite[part B]{Johnstone}, Johnstone provides an alternative approach to the construction of classifying topoi which precedes the notion of geometric theory, and for every intended application is indistinguishable from its more logical counterpart. That is, the construction of classifying topoi via \textit{weighted limits}. From a technical point of view, there are two main results in the paper. The first one concerns the $(\infty,2)$-category of topoi.

\begin{theorem*}[\Cref{thm6.1}]
    The $(\infty,2)$-category $TOP_{\mathcal{B}}$ of $\mathcal{B}$-topoi and bounded geometric morphisms has weighted limits.
\end{theorem*}

The second main result is about the existence of classifying topoi. Indeed, in the last section we deploy this technology to construct classifying topoi. For us, a \textit{theory} will just be a construction, or more precisely, a prestack

\[\mathbb{T}: TOP_{\mathcal{B}}^{op} \to CAT_{\infty}\]

which should be identified with the association mapping a topos $\mathcal{E}$ to the category of models of the intended theory $\mathbb{T}$ in $\mathcal{E}$. Of course, since no logic has been defined, this is just an informal suggestion, and this notion of theory is indeed much more general than an appropriate notion of geometric theory would be. Being more general, this notion captures a very large number of interesting constructions. For a theory to have a classifying topos means that the prestack is represented by a topos,

\[\mathbb{T} \simeq  TOP_{\mathcal{B}}(-,\mathcal{B}[\mathbb{T}]).\]

\begin{example*}[Models of a Lawvere theory]
Let $\mathsf{L}$ be a Lawvere theory in the sense of \cite{berman2020higher}. Then we have a canonically associated theory mapping a topos $\mathcal{E}$ to its $\infty$-category of $\mathsf{L}$-models in $\mathcal{E}$, $\mathbb{T}^\mathsf{L}:  TOP_{\mathcal{S}}^{op} \to CAT_{\infty}.$ 
\end{example*}

\begin{example*}[\Cref{spectra}]
For $\mathcal{E}$ a topos we can construct the category of spectrum objects in $\mathcal{E}$ and consider the associated prestack, $\mathbb{S}: TOP_{\mathcal{S}}^{op} \to CAT_{\infty}.$
\end{example*}

We show that these two prestacks both have a classifying topos and we provide a general recipe to construct classifying topoi of \textit{geometric sketches}. The notion of sketch was recently imported to $\infty$-categories in \cite{casacuberta2025sketchable} and has proven to be extremely expressive, overarching essentially every imaginable construction. In our paper we slightly restrict the focus to geometric sketches (which can be taken as a practical notion of geometric theory) and prove the following.

\begin{theorem*}[\Cref{smallgeom} and \Cref{classtoposforgeoske}]
Let $\mathbb{T}$ be a theory, then the following are equivalent:
\begin{itemize}
    \item $\mathbb{T}$  admits a classifying topos;
    \item $\mathbb{T}$ is the prestack of models of a (small) geometric sketch.
\end{itemize}
\end{theorem*}

\subsection*{Structure of the paper}
As it is evident from the introduction, most of our effort is devoted to proving \Cref{thm6.1}. To do so we essentially reproduce the original strategy for the proof from Johnstone. In \Cref{background} we establish the notation on $\infty$-categories and $(\infty,2)$-categories). In \Cref{nice2limits} we show that an $(\infty,2)$-category has weighted limits if and only if it has conical limits and powers with the arrow (\Cref{cor3.7}). It was already observed by Lurie that the $(\infty,1)$-category of topoi has all limits, and hence we are deduced to prove that we can power along the arrow. In \Cref{internal}, we briefly review the rudiments of internal higher category theory, as developed in \cite{Martini}, \cite{Martini-Wolf-1} and \cite{Martini-Wolf-2}. In \Cref{lemma}, we prove that indeed we can power via the arrow, and finally in \Cref{secboundedtopoi} we establish the appropriate $\infty$-bicategory of topoi and prove \Cref{thm6.1}. In \Cref{theobjclass}, we construct the object classifier, which is the blueprint for the whole theory of classifying topoi. Finally, in \Cref{classifying}, we discuss the notion of classifying topos  of a geometric sketch and provide many examples.

\subsection*{Acknowledgments}
 The first named author received funding from Knut and Alice Wallenberg Foundation, grant no. 2020.0199. The second named author was  supported by the Starting Grant 101077154 “Definable Algebraic Topology” from
 the European Research Council awarded to Martino Lupini.

\section{Background} \label{background}
\subsection{Background on $\infty$-categories}

\begin{notation}[$\infty$-categories]
In working with $\infty$-categories, we will often work in the context of Jacob Lurie's book \cite{LurieHTT} -see there for more details. An $\infty$-category is by definition a quasi-category, i.e. a simplicial set satisfying the appropriate lifting property. We often will write objects (or 0-simplices) in an $\infty$-category by lower case letters, such as $x, y$. We call the 1-simplices of an $\infty$-category \emph{edges} or \emph{1-morphisms}. An edge is said to be an equivalence if and only if it represents an equivalence in the \emph{homotopy category} of an $\infty$-category (see \cite[Section 1.2.3]{LurieHTT} for the definition of the homotopy category). An $\infty$-category is called an $\infty$-groupoid, iff each of its edges are invertible. In this case, the $\infty$-groupoids are precisely the Kan complexes by a theorem of Joyal (\cite{Joyal}).
\end{notation}
\begin{notation}[Morphisms in an $\infty$-category]
Given an object $x$ in an $\infty$-category $C$, we will write $C_{/x}$ and $C^{/x}$ for the slice constructions of \cite[Proposition 1.2.9.2 and Definition 4.2.1.1]{LurieHTT}, respectively. We will write $C_{x/}$  and $C^{x/}$ for their dual coslice constructions. 
Given two objects $x, y$ in an $\infty$-category $C$, we will write $Hom_{C}(x, y)$ for the space of maps between $x$ and $y$. We will be working in a relatively model-independent manner, so it does not matter which of the (equivalent) models of mapping spaces from \cite[Section 1.2.2]{LurieHTT} we use. An \emph{equivalence of $\infty$-categories} is just an equivalence in Joyal's model structure for $\infty$-categories. That is, it induces an equivalence of homotopy categories, as well as  induces weak equivalences of mapping spaces. 
    \end{notation}

\begin{notation}[Simplicial sets]
We will write $X^{K}$ for the internal hom in simplicial sets. If $X$ is an $\infty$-category, then $X^{K}$ is also an $\infty$-category and we often write $Fun(K, X)$ to emphasize that this is the \emph{$\infty$-category of functors} from $K$ to $X$. We will write $Set_{\Delta}$ for the category of simplicial sets.
\end{notation}

\begin{notation}[Simplicial categories]
    By a simplicial category, we mean a simplicially enriched category. Given a simplicial category C, we will write N(C) for its homotopy coherent nerve
(see \cite[Definition 1.1.5]{LurieHTT}). It should be noted that in the case we regard an
ordinary category as an enriched category with discrete mapping spaces, this
recovers the ordinary nerve construction. Note that we can also regard $Set_{\Delta}$ as a simplicially enriched category by the construction above. 
\end{notation}

\begin{notation}
We will write $\mathcal{S}$, $Cat_{\infty}$ for the $\infty$-categories of spaces and quasi-categories. They are obtained by applying $N$ to the full (i.e. in each simplicial degree) subcategories $Set_{\Delta}$ spanned by the $\infty$-groupoids and $\infty$-categories respectively. Given an $\infty$-category $C$, we will write $Psh(C)$ and $Psh_{Cat_{\infty}}(C)$ for $Fun(C^{op}, \mathcal{S}), Fun(C^{op}, Cat_{\infty})$, respectively. 
\end{notation}

\begin{definition}\label{def1.1}
A \emph{marked simplicial set} is an ordered pair $(X, E)$ where $X$ is a simpicial set and $E \subseteq X_{1}$. A map $f (X, E_{1}) \rightarrow (Y, E_{2})$ of marked simplicial sets is a map $f : X \rightarrow Y$ such that $f(E_{1}) \subseteq E_{2}$. We will write $Set^{+}_{\Delta}$ for the category of marked simplicial sets. This category is cartesian closed. 
\end{definition}

\begin{theorem}\label{thm1.2} (see \cite{LurieHTT})
There exists a model category structure on the category $Set_{\Delta}^{+}$ of marked simplicial sets in which the cofibrations are the monomorphisms and the fibrant objects are the marked simplicial sets.
\end{theorem}

\subsection{Background on $\infty$-bicategories}

The purpose of this section is to provide some brief background on the theory of $\infty$-bicategories from \cite{Gagna-Harpaz-Lanari} and \cite{LurieGoodwillie}.

\begin{definition}\label{def2.1}
A \emph{scaled simplicial set} is an ordered pair $(X, T_{X})$ consisting of a simplicial set $X$ and a subset $T_{X}$ of the 2-simplices of $X$ containing the degenerate ones. We call $T_{X}$ the \emph{thin triangles} in $X$.  A map of scaled simplicial sets $f: (X, T_{X}) \rightarrow (Y, T_{Y})$ is a map of simplicial sets $f : X \rightarrow Y$ such that $f(T_{X}) \subseteq T_{Y}$.
We denote by $Set_{\Delta}^{sc}$ the category of scaled simplicial sets. 
\end{definition}

\begin{construction}\label{con2.2}
Let $X$ be a simplicial set. We denote by $X_{\flat} = (X, deg_{2}X)$ the scaled simplicial set whose thin simplices are exactly the degenerate simplices. We denote by $X_{\sharp} := (X, X_2)$ the scaled simplicial set which has as its thin simplices all the 2-simplices of $X$.  
The assignments 
$$
X \mapsto X_{\flat}, \, \, \, \, X \mapsto X_{\sharp}
$$
are left and right adjoint to the forgetful functor
$$
Set_{\Delta}^{sc} \rightarrow Set_{\Delta}.
$$
\end{construction}

\begin{definition}\label{def2.3}
Given a scaled simplicial set $X$, we define its core to be the simplicial set $X^{th}$ spanned by those n-simplices of $X$ whose 2-dimensional faces are thin triangles (note that here $th$ stands for thin). The assignment 
$X \mapsto X^{th}$ is right adjoint to the forgetful functor $(-)_{\sharp}$.
\end{definition}

\begin{definition}\label{def2.4}
We say that a scaled simplicial set $C$ is an $\infty$-bicategory if the map $C \rightarrow *$ has the right lifting property with respect to the class of maps $\mathbf{S}$ from \cite[Definition 1.2.6]{Gagna-Harpaz-Lanari}.
\end{definition}

\begin{theorem}\label{thm2.5}
There is a model structure on $Set_{\Delta}^{sc}$ in which the cofibrations are the monomorphisms and the fibrant objects are the $\infty$-bicategories. 
\end{theorem}

\begin{remark}\label{rmk2.6} (see \cite[Remark 1.2.9]{Gagna-Harpaz-Lanari})
Given an $\infty$-bicategory $C$, $C^{th}$ is an $\infty$-category. 
\end{remark}

\begin{remark}\label{rmk2.7}
As noted in \cite{LurieGoodwillie}, the cartesian product is a left Quillen bifunctor. In particular, for any two scaled simplicial sets $X, Y$, we have a mapping object $Fun(X, Y)$, which satisfies 

$$
Hom_{Set_{\Delta}^{sc}}(Z, Fun(X, Y)) \cong Hom_{Set_{\Delta}^{sc}}(Z \times X, Y).
$$
Furthermore, we have that if $C$ is an $\infty$-bicategory, so is $Fun(K, C)$ for any scaled simplicial set $K$.

\end{remark}

\begin{construction}\label{con2.8} 
We note that there is a Quillen adjunction 
$$
\mathfrak{C}^{sc} :  Set^{sc}_{\Delta} \leftrightarrows Set^{+}_{\Delta} : N^{sc}
$$
where the right adjoint is called the scaled coherent nerve. Applying the scaled coherent nerve to the full subcategory $Set^{+}_{\Delta}$ spanned by $\infty$-categories, we obtain the $\infty$-bicategory of $\infty$-categories, which we denote by $CAT_{\infty}$. Note that $CAT_{\infty}^{th} = Cat_{\infty}$.
\end{construction}

\begin{definition}\label{def2.9}
Suppose that $C$ is an $\infty$-bicategory. Then for each $x, y \in C_{0}$, we will write $Hom_{C}(x, y)$ for the simplicial set whose $n$-simplices are maps 
$s : \Delta^{n} \times \Delta^{1} \rightarrow  C$ such that $s|_{\Delta^{n} \times \{ 0\} } = x, s|_{\Delta^{n} \times \{ 1\} } = y$ and $s|_{\Delta^{\{(i, 0), (i, 1), (j, 1) \}}}$ is thin.

\end{definition}

\begin{definition}\label{def2.10}
We will write $\mathbf{i} : Cat_{\infty} \rightarrow \mathcal{S}$ for the right adjoint of the inclusion $\mathcal{S} \subseteq Cat_{\infty}$. We call $\mathbf{i}$ the \emph{core} functor. 
\end{definition}

\begin{lemma}\label{lem2.11}
Suppose that $\mathcal{C}$ is an $\infty$-bicategory. Then we have the following:
\begin{enumerate}
\item{$Hom_{C}(x, y)$ is an $\infty$-category for each $x, y \in C$.}
\item{We have an identification $\mathbf{i} Hom_{C}(x, y) = Hom_{C^{th}}(x, y)$. }
\end{enumerate}
\end{lemma}
\begin{proof}
The first statement is from \cite[Notation 1.2.13]{Gagna-Harpaz-Lanari}.

For the second statement, we first want to show that there is a natural inclusion $\mathbf{i} Hom_{C}(x, y) \subseteq Hom_{C^{th}}(x, y)$, where $Hom_{\mathcal{C}^{th}}$ is the model of mapping space in an $\infty$-category described in \cite[pg. 28]{LurieHTT}.
To do this, it suffices to show that each map $s : \Delta^{n} \times \Delta^{1} \rightarrow C$ representing an $n$-simplex of $\mathbf{i} Hom_{C}(x, y) $ factors through $C^{th}$. In other words, we want to show that each 2-simplex of $\Delta^{n} \times \Delta^{1}$ maps under $s$ to a thin 2-simplex. Since each degenerate 2-simplex is automatically thin, it suffices to show that any map $e := \Delta^{1} \times \Delta^{1} \subseteq \Delta^{n} \times \Delta^{1} \xrightarrow{F} C$ corresponding to an edge  $\mathbf{i} Hom_{C}(x, y)$, maps triangles to thin triangles. By definition, $e|_{\Delta^{1} \times \{0 \}}, e|_{\Delta^{1} \times \{ 1\}}$ are both equivalences. Thus, by the construction of mapping spaces from \cite[Section 1.2.2]{LurieHTT} , $e$ maps each edge of $\Delta^{1} \times \Delta^{1}$ to equivalences. Since triangles spanned by equivalences are thin in an $\infty$-bicategory, we have that $e$ factors through $C^{th}$ as required.

We now want to show the reverse inclusion $ Hom_{C^{th}}(x, y) \subseteq \mathbf{i} Hom_{C}(x, y)$. Arguing as before, it suffices to show that each edge of $Hom_{C^{th}}(x, y)$ is an edge of $ \mathbf{i} Hom_{C}(x, y)$. 
Since $Hom_{C^{th}}(x, y)$ is a Kan complex, every edge of  $Hom_{C^{th}}(x, y)$ is an equivalence, and thus corresponds to a map $\Delta^{1} \times \Delta^{1} \rightarrow C$ which takes edges to equivalences. Thus, since triangles in an $\infty$-bicategory whose edges are equivalences are thin, $e$ represents an edge of $\mathbf{i}Hom_{C}(x, y)$. 
\end{proof}

\begin{construction}\label{con2.12} (see \cite[Definition 4.2.1]{Gagna-Harpaz-Lanari})
Suppose that $C$ is an $\infty$-bicategory and $x \in \mathcal{C}$. Then the slice  $C^{/x}$ is an $\infty$-bicategory such that the n-simplices of $(C^{/x})^{th}$ can be identified with maps $\phi: \Delta^{1} \times \Delta^{n} \rightarrow C^{th}$ with $\phi|_{\{ 1\} \times \Delta^{n}} = x$ and the thin simplices can be identified with maps 
$$
\Delta^{1}_{\flat} \otimes \Delta^{2}_{\sharp} \rightarrow C
$$
where $\otimes$ is the Gray tensor product of \cite[Definition 4.1.1]{Gagna-Harpaz-Lanari}.
Thus, by the definition of the Gray tensor product, there is a natural inclusion $C^{/x} \rightarrow Fun(\Delta^{1}_{\flat}, C)$.

Similarly, one can define the coslice $C^{x/}$, so that there are natural inclusions:
$$
C^{/x} \rightarrow Fun(\Delta^{1}_{\flat}, C) \leftarrow C^{x/}
$$

\end{construction}

\section{$(\infty, 2)$-Categorical Limits from Conical Limits and Arrows}
\label{nice2limits}
The purpose of this section is to briefly review the theory of weighted limits in an $\infty$-bicategory from \cite{Gagna-Harpaz-Lanari}. Then we will prove a result which says that if an $\infty$-bicategory $C$ has cotensors with $\Delta^{1}$ and the underlying $\infty$-category has all limits, then the $C$ has all weighted limits. This result and its proof were inspired by a classical 2-categorical theorem (see \cite[Section 2.1]{Kelly-Limits}).

\begin{definition}\label{def3.1}
Let $C$ be an $\infty$-bicategory and $f : J \rightarrow C$ and $w : J \rightarrow CAT_{\infty}$ be two functors of $\infty$-bicategories. Then the limit $l \in C$ of $f$  weighted by $w$ is characterized by the natural equivalence of $\infty$-categories: 
\begin{equation}\label{eqxxxx}
Hom_{C}(x, l) \cong Nat_{J}(w, Hom_{C}(x, f(-))),
\end{equation}
where the right-hand side denotes the mapping category in the $\infty$-bicategory $Fun(J, CAT_{\infty})$ between $w$ and $Hom_{C}(x, f(-)))$.

We say that a limit is \emph{conical} if the weight functor is the constant functor. We say that an $\infty$-category is finite if it can be written as the homotopy colimit of finitely many simplices. We say that a limit is finite if the underlying $\infty$-category of $J$ is finite and $w(j)$ is a finite $\infty$-category for each $j \in J$.

\end{definition}

\begin{remark}\label{rmk3.2}
Note that the underlying $\infty$-category of $Fun(J, CAT_{\infty})$ is $Fun(J^{th}, Cat_{\infty})$.
Thus, by \cref{lem2.11} above applying $\mathbf{i}$ to both sides of the \cref{eqxxxx} yields an equivalence of $\infty$-categories,
$$
Hom_{C^{th}}(x, l) \cong Nat_{J^{th}}(w^{th}, Hom_{C^{th}}(x, f^{th}(-)))
$$
which we say is the underlying $\infty$-categorical universal property of the limit.  

By the characterization of weighted limits in $\infty$-categories from \cite[Proposition 5.1]{Gepner-Haugseng-Nikolaus} and \cite[Definition 2.8]{Gepner-Haugseng-Nikolaus}, this equivalence exhibits $l$ as the limit of $f^{th}$ weighted by $w^{th}$. 
\end{remark}

\begin{lemma}\label{lem3.3}
Suppose that $X$ is a complete $\infty$-category. Then $X$ is cotensored over $Cat_{\infty}$ if it admits a cotensoring with  $\Delta^{1}$.
\end{lemma}

\begin{proof}
Consider the full subcategory $\{ \Delta^{1} \} \subseteq Cat_{\infty}$ on the interval $\Delta^{1}$. We will define the cotensoring to be the right Kan extension
$$
\xymatrix
{
(\{ \Delta_{1} \})^{op} \times X \ar[rr]^{   (-)^{\Delta^{1}  }} \ar[d] && X \\
Cat_{\infty}^{op} \times X \ar@{.>}[urr]_{(-)^{(-)}} &&  
}
$$
By \cite[Lemma 4.3.2.13]{LurieHTT} and the completeness of $X$, the right Kan extension exists. 

By \cite[Proposition 4.2.3.8]{LurieHTT}, we can regard as $F$ as the right Kan extension of $(-) ^{  (-)}|_{\Delta \times X}$. Note that by \cite[Lemma 2.5]{Haugseng}, $\Delta$ is a dense subcategory of $Cat_{\infty}$ in the sense that 
for each $x \in Cat_{\infty}$, $x$ is the colimit of 
$$
\Delta_{/x} \rightarrow \Delta \subseteq Cat_{\infty}.
$$ 
Thus, by the formula for Kan extensions (as in \cite[Proposition 4.2.3.8]{LurieHTT}), and the fact that $Hom_{X}(-, d)$ sends colimits to limits
it suffices to show that we have a isomorphism 
$$
Hom_{X}(c, d^{\Delta^{n}})  \cong Hom_{Cat_{\infty}}(\Delta^{n}, Hom_{X}(c, d))
$$
natural in $\Delta^{n}, c, d$.

However, appealing to \cite[Proposition 4.2.3.8]{LurieHTT} again, we note that $(-)^{(-)}|_{\Delta \times X}$ is the right Kan extension of $ (-)^{\Delta^{1}} $ along $\{ \Delta^{1}\} \subseteq \Delta$. Thus, reasoning as before, it suffices to show that $\Delta^{1}$ is dense in $\Delta$. But the fact that $\Delta^{n}$ is the colimit of 
$$
(\{ \Delta^{1} \})_{/\Delta^{n}} \rightarrow \{ \Delta^{1} \} \subseteq \Delta
$$
 can be reduced to statement about the corresponding ordinary categories $\Delta, \{ \Delta^{1} \} \subseteq \Delta $ which is something that can be easily verified by hand. 

\end{proof}

\begin{lemma}\label{lem3.4}
Suppose that $C$ is an $\infty$-bicategory. Suppose that
\begin{enumerate}
\item{$C$ has a tensor with $\Delta^{1}$}
\item{$C^{th}$ has all limits.}
\end{enumerate}
Then $C$ has all (weighted) limits. 
\end{lemma}

\begin{proof}
Suppose that $F : J \rightarrow C$ is a diagram of $\infty$-bicategories and $w : J \rightarrow CAT_{\infty}$ be a weighting. 
By hypothesis, and \cref{rmk3.2} above we have an equivalence, natural in $x$

\begin{equation}\label{ijig}
\phi_{x} : \mathbf{i} Hom_{C}(x, l) \simeq \mathbf{i}Nat_{J}(w, Hom_{C}(x, F(-)))
\end{equation}
Postcomposition by the counit of the adjunction yields a map 
$$
\psi_{x} : Hom_{C}(x, l) \rightarrow Nat_{J}(w, Hom_{C}(x, F(-)))
$$
which we want to show is an equivalence of $\infty$-categories. 

By using the naturality of \cref{ijig},  we have a commutative diagram for $i = 0, 1$
$$
\xymatrix
{
\mathbf{i} Hom_{C}(x \otimes \Delta^{1}, l)  \ar[d]^{x \otimes d^{i} }  \ar[rr]_>>>>>>>>>>>{\phi_{x \otimes \Delta^{1}}} && \mathbf{i}Nat_{J}(w \otimes \Delta^{1}, Hom_{C}(x, F(-))) \ar[d]^{ w \otimes d^{i} } \\
\mathbf{i} Hom_{C}(x \otimes \Delta^{0}, l)  \ar[rr]_>>>>>>>>>>{\phi_{x}} &&  \mathbf{i}Nat_{J}(w \otimes \Delta^{0},  Hom_{C}(x, F(-)))
}
$$

which by the universal property of tensors, can be identified with a commutative diagram: 
$$
\xymatrix
{
\mathbf{i}(Hom_{C}(x, l)^{\Delta^{1}}) \ar[d] \ar[rr]_>>>>>>>>>>{\mathbf{i}(\psi_{x}^{\Delta^{1}}) } && \mathbf{i}(Nat_{J}(w, Hom_{C}(x, F(-)))^{\Delta^{1}}) \ar[d] \\
\mathbf{i}(Hom_{C}(x, l)) \ar[rr]_>>>>>>>>>>>>{\mathbf{i}\psi_{x}} && \mathbf{i}(Nat_{J}(w, Hom_{C}(x, F(-)))) 
}
$$

For brevity, let $M_{x} = Hom_{C}(x, l), N_{x} = Nat_{J}(w, Hom_{C}(x, F(-)))$. Let $a, b \in M_{x}$. Consider the commutative diagram
$$
\xymatrix
{
 \mathbf{i}(M_{x}^{\Delta^{1}}) \ar[rr]_{((-)^{d^{0}}, (-) ^{d^{1}}) } \ar[d]_{\mathbf{i} \psi_{x}^{\Delta^{1}} }  && \mathbf{i}M_{x} \times \mathbf{i}M_{x} \ar[d]^{\mathbf{i} \psi_{x} \times \mathbf{i} \psi_{x}  } &  \ar[l]_>>>>>{(a, b)} \, * \ar[d]^{id} \\
\mathbf{i}(N_{x}^{\Delta^{1}}) \ar[rr]_{(-)^{d^{0}}, (-) ^{d^{1}})} && \mathbf{i}N_{x} \times \mathbf{i}N_{x}  & \ar[l]_>>>>>{(a, b)}  \, *
}
$$
which induces an equivalence of pullbacks. By the description of mapping spaces in \cite[Section 1.2.2]{LurieHTT}, and the fact that $\mathbf{i}$ preserves pullbacks being a right adjoint, this can be identified with the induced map 
$$
\psi_{x} : Hom_{M_{x}}(a, b) \rightarrow Hom_{N_{x}}(\psi_{x}(a), \psi_{y}(a))
$$
Since $\psi_{x}$ induces equivalences of mapping spaces and $\mathbf{i} \psi_{x}$ is an equivalence of $\infty$-groupoids, $\psi_{x}$ is an equivalence of $\infty$-categories, as required.

\end{proof}

\begin{corollary}\label{co3.5}
Suppose that $C$ is an $\infty$-bicategory that has all limits indexed by $\infty$-categories, and a functor $(-)^{\Delta^{1}} : C \rightarrow C$. Then $X$ is cotensored over $CAT_{\infty}$. 
\end{corollary}

\begin{proof}
The argument of \cref{lem3.3} shows that we have a functor 
$$
C \times CAT_{\infty} \rightarrow C
$$ 
given by 
$$
x^{C} = \underset{   (\{ \Delta^{1}\}^{ /C}) \xrightarrow{p} CAT_{\infty}}{lim} x^{\Delta^{1}}
$$
 $(-)^{/ C}$ is the slice construction from \cref{con2.12}  $p$ is the projection. We need to verify that this cotensor product satisfies the appropriate universal property. By the representable nature of limits 
and the fact that $CAT_{\infty}$ has tensors with $\Delta^{1}$, \cref{lem3.4} implies that it suffices to show that the underlying functor of $\infty$-categories satisfies the appropriate universal property. However, the underlying $\infty$-category of $C^{/X}$ is $(C^{th})^{/X}$ by the construction of \cite[Section 4.1]{Gagna-Harpaz-Lanari}. By \cite[Proposition 4.2.1.5]{LurieHTT}, the value of the functor on underlying $\infty$-categories is equivalent to the limit of the composite: 
$$
(C^{th})_{/X} \rightarrow (C^{th})^{/X} \xrightarrow{p^{th}} Cat_{\infty} 
$$
where the first map is the natural map from \cite[Proposition 4.2.1.5]{LurieHTT}. 
Thus, the construction reduces to the of the construction of \cref{lem3.3} on underlying $\infty$-categories, hence the result.  

\end{proof}

\begin{theorem}\label{thm3.6}
Suppose that $C$ is an $\infty$-bicategory. Then if $C$ has a powering $(-)^{\Delta^{1}}$ and $C^{th}$ has all (unweighted) limits, then $C$ has all (weighted) limits
\end{theorem}
\begin{proof}

By the representable nature of limits, we are reduced to proving the statement for $CAT_{\infty}$. In this case, by \cref{lem3.4}, we want to show that weighted limits in the underlying $\infty$-category can be computed as ordinary limits. Since $Cat_{\infty}$ is cotensored over itself this follows from the formulas for weighted colimits from \cite{Gepner-Haugseng-Nikolaus} in terms of the twisted arrow category. 

\end{proof}

\begin{corollary}\label{cor3.7}
Suppose that $C$ is an $\infty$-category. Suppose that $C^{th}$ has all finite limits and a cotensoring $(-)^{\Delta^{1}}$. Then $C$ has all finite weighted limits. 
\end{corollary}

\begin{proof}

The representable nature of limits allows us to reduce to the case of weighted limits on the underlying $\infty$-category of $C$. 

By \cite{Haugseng}
A finite limit of $F :J \rightarrow C^{th}$ weighted by $w : J \rightarrow Cat_{\infty}$ can be computed as the limit on underlying $\infty$-categories of
$$
\Delta_{/J} \rightarrow J \times J^{op} \xrightarrow{p} (J) \times (J)^{op} \xrightarrow{F(-)^{w(-)}}  \mathcal{C} 
$$
Since $\Delta_{/J}$ is finite when $J$ is, finite weighted limits in the underlying $\infty$-category exist in the presence of ordinary finite limits and cotensors with finite $\infty$-categories. However, the latter exist by the arguments of \ref{lem3.3}.

\end{proof}

\section{Internal Higher Category Theory in an $\infty$-topos} \label{internal}

In this section, we will briefly review the rudiments of internal higher category theory, as developed in \cite{Martini}, \cite{Martini-Wolf-1} and \cite{Martini-Wolf-2}.

\begin{definition}\label{def4.1}
We call the inclusion $I_{n} : \Delta^{1} \times_{\Delta^{0}} \Delta^{1} \cdots \times_{\Delta^{0}} \Delta^{1} \rightarrow \Delta^{n}$ the spine inclusion. We will define 
$$
E_{1} = \Delta^{0} \coprod_{\Delta^{\{ 0, 2\}}} \Delta^{3} \coprod_{\Delta^{\{ 1, 3\}}} \Delta^{0}.
$$
This is known as the walking equivalence.
\end{definition}

\begin{definition}\label{def4.2}
Suppose that $\mathcal{B}$ is an $\infty$-topos. We will write $\mathcal{B}_{\Delta} := Fun(\Delta^{op}, \mathcal{B})$ for the $\infty$-topos of simplicial objects in $\mathcal{B}$. We will call an object $C \in \mathcal{B}_{\Delta}$ a $\mathcal{B}$-category if the maps
\begin{enumerate}
\item{$C_{n} \rightarrow C_{1} \times_{C_{0}} C_{1} \cdots \times_{C_{0}} C_{1} $}
\item{$C_{3} \rightarrow C_{1} \times C_{1} \times_{C_{0} \times C_{0}} C_{3} $}
\end{enumerate}
induced by the spine inclusion and the map $E^{1} \rightarrow \Delta^{0}$ are equivalences of $\infty$-categories. We will denote by $\mathrm{Cat}(\mathcal{B})$ the full subcategory of $\mathcal{B}_{\Delta}$, whose objects are the $\mathcal{B}$-categories.

Each object $B \in \mathcal{B}$ will be identified with a constant simplicial object, which is evidently a $\mathcal{B}$-category. 

\end{definition}

\begin{remark}\label{rmk4.3}
In \cite[Definition 3.1.5 and Definition 3.2.1]{Martini}, a $\mathcal{B}$-category is defined as an object of $\mathcal{B}_{\Delta}$ which is \emph{internally local} to the maps $I_{n} $ and $E^{1} \rightarrow \Delta^{0}$, and \cite[Proposition 3.2.7]{Martini} shows that this definition is equivalent to the one given above. 

\end{remark}

\begin{theorem}\label{equivalencewsheaves} (\cite[Proposition 1.2.2.1]{Martini-Wolf-2})
Let $Sh_{Cat_{\infty}}(\mathcal{B})$ denote the category of limit preserving functors $\mathcal{B}^{op} \rightarrow Cat_{\infty}$, we call this the category $\mathcal{B}$-valued sheaves. Then there is an equivalence of $\infty$-categories $Sh_{Cat_{\infty}}(\mathcal{B}) \simeq Cat(\mathcal{B})$. 
\end{theorem}

\begin{construction}\label{con4.4} (see \cite{Martini-Wolf-2})
By \cite[Proposition 1.2.1.4]{Martini-Wolf-2}, $Cat(\mathcal{B})$ has an internal hom, which we denote by 
$$
\underline{Fun}_{\mathcal{B}}(X, Y).
$$

Given a geometric morphism $f_{*} : \mathcal{A} \leftrightarrows  \mathcal{B} : f^{*}$ of $\infty$-topoi, we have an  adjunction induced by precomposition:
$$
Cat(\mathcal{A}) \leftrightarrows Cat(\mathcal{B}).
$$
In particular, if we consider the terminal map $\mathcal{B} \rightarrow *$, we obtain an adjunction of $\infty$-categories, 
$$
const_{\mathcal{B}} : Cat_{\infty} \leftrightarrows  Cat(\mathcal{B}) : \Gamma_{\mathcal{B}}.
$$
This allows us to define bifunctors
\begin{enumerate}
\item{(Functor $\infty$-category) $Fun_{\mathcal{B}}(X, Y) := \Gamma_{\mathcal{B}} \underline{Fun}_{\mathcal{B}}(X, Y)$}
\item{(Cotensoring) $(-)^{(-)} := Fun_{\mathcal{B}}(const_{\mathcal{B}}(-), (-))$}
\end{enumerate}
We will often implicitly identify a simplicial set $S$ with $const_{\mathcal{B}}(S)$. 

\end{construction}

\begin{construction}\label{con4.5}
We will write $CAT_{\mathcal{B}}$ for the $\infty$-bicategory from \cite[Remark 3.4.3]{Martini}, whose objects are $\mathcal{B}$-categories and such that 
$Hom_{CAT_{\mathcal{B}}}(X, Y) = Fun_{\mathcal{B}}(X, Y)$.  

\end{construction}

\begin{definition}\label{def4.6} (see \cite[Proposition 1.2.11.2 and Definition 1.2.3.1]{Martini-Wolf-2}).

Let $C$ be a $\mathcal{B}$-category. If $a \in C$, we call $c : a \rightarrow C$ an object of $C$ in context $a$.
Given an object $a \in \mathcal{B}$, the \emph{mapping $\mathcal{B}_{/A}$-groupoid} associated to $c, d : a \rightrightarrows C$ is defined to be the pullback 
$$
\xymatrix
{
Hom_{C}(c, d) \ar[r] \ar[d] & C_{1} \ar[d]^{(d_{0}, d_{1})} \\
A \ar[r]_{(c, d)} & C_{0} \times C_{0}
}
$$
An adjunction between $\mathcal{B}$-categories is an equivalence of functors
$$
Hom_{D}(l(-), (-)) \cong map_{C}((-), r(-)).
$$
\end{definition}

\begin{definition}\label{def4.7}
Let $I, C$ be $\mathcal{B}$-categories. 
A colimit functor is a left adjoint to the diagonal map $diag: \mathcal{B} \rightarrow \underline{Fun}_{\mathcal{B}}(I, \mathcal{B})$. There is an evident dual notion of limit functor. 
A \emph{colimit} for a diagram $d : A \rightarrow \underline{Fun}_{\mathcal{B}}(I, C)$ is a corepresenting object of $map_{\underline{Fun}_{\mathcal{B}}(I, C)}(d, diag)$. Similarly, we can define a limit of diagram as a representing object. 
\end{definition}

\begin{construction}\label{con4.8}

By \cite[1.2.6]{Martini-Wolf-2}, there is a (large) $\mathcal{B}$-category, called the \emph{universe}, which under the equivalence of \cref{equivalencewsheaves}, is given by the sheaf 
$$
\Omega(A) \simeq \mathcal{B}_{/A}.
$$
Given two objects of $\mathcal{B}_{/A}$, which can be viewed as objects of $\Omega$ in context $A$, we have an equivalence
$$
map_{\Omega}(G, H) \simeq Fun_{\mathcal{B}_{/A}}(G, H).
$$ 

For a $\mathcal{B}$-category $C$, we write $\underline{Psh}_{\mathcal{B}}(C) := \underline{Fun}_{\mathcal{B}}(C, \Omega_{\mathcal{B}})$.
\end{construction}

\begin{definition}\label{def4.9}
We say that a large $\mathcal{B}$-category  is presentable if it is equivalent to a $\mathcal{B}$-categorical localization of $\underline{Psh}_{\mathcal{B}}(C)$ for some small $\mathcal{B}$-category $C$ (see \cite[Section 2.4]{Martini-Wolf-2} for more details).
We say that a $\mathcal{B}$-category $C$ is a $\mathcal{B}$-topos if it satisfies the following conditions: 
\begin{enumerate}
\item{It is presentable.}
\item{It satisfies descent. That is, the slice functor $C_{/-} : C \rightarrow Cat_{\mathcal{B}}$ is cocontinuous (see \cite[Section 3.1]{Martini-Wolf-2} for more details).}
\end{enumerate}

We say that a morphism of $\mathcal{B}$-topoi is \emph{algebraic} if preserves colimits and finite limits. We will write $\underline{Fun}^{alg}_{\mathcal{B}}(X, Y)$ for the full subcategory of $\underline{Fun}_{\mathcal{B}}(X, Y)$ on algebraic morphisms and $Fun^{alg}_{\mathcal{B}}(X, Y)$ for the global sections of  $\underline{Fun}^{alg}(X, Y)$.

We will write $TOP_{\mathcal{B}}$ for the $\infty$-bicategory of $\mathcal{B}$-topoi. That is, the $\infty$-bicategory whose objects are $\mathcal{B}$-topoi and such that the $\infty$-category of morphisms between $X, Y$ is given by $Fun_{\mathcal{B}}^{alg}(X, Y)$. We will write $Top_{\mathcal{B}} : = TOP_{\mathcal{B}}^{th}$.
\end{definition}

\begin{remark}\label{rmk4.10}
The adjoint functor theorem for $\mathcal{B}$-categories (\cite[Proposition 2.4.3.1]{Martini-Wolf-2}) implies that every algebraic morphism has a right adjoint. The right adjoint of an algebraic morphism is called a \emph{geometric morphism}.  
\end{remark}

The following is an important characterization of $\mathcal{B}$-topoi.

\begin{theorem}\label{thm4.11} (see \cite[Theorem 3.2.5.1]{Martini-Wolf-2})
 There is an equivalence of $\infty$-categories:
$$
Top_{\mathcal{B}} \simeq Top_{/\mathcal{B}}^{alg}
$$
given by taking global sections. 
\end{theorem}

\begin{remark}\label{rmk4.13}
In \cite{Johnstone}, the morphisms in the 2-category of $\mathcal{B}$-topoi are defined to be geometric morphisms, as opposed to algebraic morphisms. Thus, $TOP_{\mathcal{B}}$ is an $\infty$-categorical analogue of the opposite of the 2-category of $\mathcal{B}$-topoi defined in \cite{Johnstone}.
\end{remark}

\section{A Lemma on Flat Functors} \label{lemma}

The purpose of this section is to construct a natural map $Fun(C^{\Delta^{1}}, \mathcal{S}) \rightarrow Fun(C, \mathcal{S})^{\Delta^{1}}$ and show that this (co)-restricts to an equivalence 
$$
Filt((C)^{\Delta^{1}}, \mathcal{S}) \simeq Filt((C), \mathcal{S})^{\Delta^{1}}.
$$
This is an special case of a theorem for filtered functors of $\mathcal{B}$-categories that will be necessary to prove the main result of the final section. 

% {\color{red} Here we should cite (just for reference) a bunch of stuff. Mostly \cite{henry2023does} and some of the papers he refers to. From a pedagogical point of view, we should also declare that our use of this fact will be to construct arrow-objects in topoi.}

\begin{definition}
We say that an $\infty$-category $C$ is filtered if for any finite simplicial set $K$, a diagram $K \rightarrow C$ can be extended to $K^{\triangleright} \rightarrow C$.
We say that a functor $F : A \rightarrow B$ of $\infty$-categories is \emph{filtered} if for any right fibration $B' \rightarrow B$ with $B'$ filtered the pullback $B' \times_{B} A$ is a filtered $\infty$-category. We will functor $Filt(C, D)$ for the full subcategory of $Fun(C, D)$ whose objects are filtered functors.
\end{definition}

\begin{theorem}\label{thm5.2}
There is a natural equivalence 
$$
Filt((C^{op})^{\Delta^{1}}, \mathcal{S}) \simeq Filt(C^{op}, \mathcal{S})^{\Delta^{1}}
$$
of functors natural in $\infty$-categories $C$
\end{theorem}

\begin{lemma}\label{lemxxx}
Suppose that  $S : C^{op}\rightarrow \mathcal{S}$ is a filtered functor of $\infty$-categories. Then the following are true for $R = \pi_{0}S^{K}$:

\begin{enumerate}
\item{There exists some $c \in C$ such that $S(c)$ is nonempty.}
\item{For any $x_{i} \in R(c_{i}), i = 1, 2$, there exists morphisms $f_{i} :  e \rightarrow c_{i},  i = 1, 2$ and $x_{3} \in R(e)$ such that $R(f_{i})(x_{3}) = x_{i} \, \, i = 1, 2$.}
\item{Suppose that we have morphisms $f_{i} : d \rightarrow c$ in $C$ and $x \in X(d)$. Then there exists a $g : e \rightarrow d$ and $y \in X(e)$ such that $g \circ f_{1} = g \circ f_{2}$ and $g(y) = x$. }
\end{enumerate}
In particular, $\pi_{0}(S)$ is a filtered functor of sets in the sense of \cite{Johnstone}.
\end{lemma}

\begin{proof}

Consider the unstraightening  $\mathbf{Un}(S) : D \rightarrow C$ of $C \rightarrow \mathcal{S}$. Then $D$ is filtered, since $\mathbf{Un}(S)$ is the pullback of $S$ along the universal right fibration (see \cite[Section 3.3]{LurieHTT}).

Since $D$ is filtered, this means that some fiber of $D \rightarrow C$ is inhabited, so that by the properties of (un)straightening adjunction there exists some $c \in C$ such that $S(c) \neq \emptyset$. Thus, $S(c)^{K} \neq \emptyset$ and we have proven the first point. 

For the second point, suppose that we choose 
$$
 (x_{1}, x_{2}) : K \coprod K \rightarrow D
$$
Then we can find an extension 
$$
(K \coprod K)^{\triangleright} \rightarrow D 
$$
which determines a homotopy 
$$
H : (K \coprod K ) \times \Delta^{1} \rightarrow D
$$
to a constant diagram $(e, e) : K \coprod K \rightarrow D$. 

By \cite[Lemma 3.2.1.9]{LurieHTT}, we can find a lift in the diagram of marked simplicial sets
$$
\xymatrix
{
(K \coprod K) \times \{ 1\})_{\flat} \ar[rr]^>>>>>>>>>{(e, e)} \ar[d] && (D, D_{cart}) \ar[d] \\
( (K \coprod K) \times \Delta^{1}, E) \ar[rr]_>>>>>>>>{Un(S) 
\circ H } \ar@{.>}[urr]_{\phi} && S^{\sharp} 
}
$$
where the cartesian edges are marked in $D$, and where $E$ is the union of the degenerate edges and the edges of the form $\{a \} \times \Delta^{1}$.

Now, choose equivalences $\phi_{i}: S(c_{i}) \simeq \mathbf{Un}(S)^{-1}(c_{i})$, $(y_{1}, y_{2}) = \phi|_{(K \coprod K) 
\times \{ 0\}}$. Let $f_{i} = Un(S) \circ H|_{(K \coprod K) \times \{ i\}} $. By the definition of unstraightening, we have $S(f_{i})(e)$ is pointwise equivalent in $S(c_{i})^{K}$ to $ y_{i}$. In particular, \cite[Corollary 5.1.2.3]{LurieHTT} implies that $\pi_{0} S^{K}(f_{i})(e_{i}) = y_{i}$.

Thus, it remains to show that $x_{i}$ and $y_{i}$ are equivalent in a fiber of $Un(S)^{K}$.
It suffices to find a lift in the diagram
$$
\xymatrix
{
\partial \Delta^{1} \ar[rr]_{(x_{i}, y_{i})} \ar[d] && D^{K} \ar[d] \\
\Delta^{1} \ar[rr]_{id_{(S)^{K} \circ x_{i}}} \ar@{.>}[urr] && (S)^{K} .
}
$$
But it follows from \cite[Proposition 3.1.2.1]{LurieHTT} that the right hand vertical map is a cartesian fibration, so we can find a cartesian lift. 

The argument for the third point is similar to the second. 

\end{proof}

\begin{lemma}\label{lem5.4} (\cite[Corollary 4.2]{Haugseng})
There is an isomorphism 
$$
\int^{f} Hom_{C}(-, dom(f)) \times F(f)  \cong F(-)
$$
natural in $F \in Fun(C^{op}, D)$.
\end{lemma}

\begin{construction}\label{con5.5}
In this construction, we are going to define for an $\infty$-category $C$, a natural map $Fun(C^{\Delta^{1}}, D) \rightarrow Fun(C, D)^{\Delta^{1}}$. 
Consider the functors $dom, codom: C^{\Delta^{1}} \rightarrow C$ given by evaluation at $0, 1$. Then there is a functor 
$$
\mathbf{T} : Fun(C^{\Delta^{1}}, D) \rightarrow Fun(C, D)^{\Delta^{1}}
$$
which takes $U $ a natural transformation $\mathbf{T}(U) := \mathbf{T}_{0}(U) \rightarrow \mathbf{T}_{1}(U)$ given by the formula: 
$$
\mathbf{T}_{0}(U) = \int^{f} Hom_{C}(-, dom(f)) \times U(f) \rightarrow \int^{f} Hom_{C}(-, codom(f)) \times U(f) = \mathbf{T}_{1}(U)
$$
where the natural transformation is induced by the evaluation map $ev : \Delta^{1} \times C^{\Delta^{1}} \rightarrow C$. 

By \cite[Theorem 5.1]{Gepner-Haugseng-Nikolaus}, $\mathbf{T}_{i}(U)$ can be described as the left Kan extension of $U$ along $ C^{\Delta^{1}} \xrightarrow{C^{d^{i}}} C^{\Delta^{0}}$. 

\end{construction}

\begin{lemma}\label{lemxx}
suppose that $F : I \rightarrow Cat_{\infty}$ is a diagram. Let $I_{\neq \emptyset}$ be the full subcategory of $I$ consisting of objects $i$ such that $F(i) \neq \emptyset$. The $colim F = colim F|_{I_{\neq \emptyset}}$
\end{lemma}

\begin{proof}
To see this, we use \cite[Theorem 4.2.4.1]{LurieHTT} to reduce to proving the corresponding statement for homotopy colimits of diagrams $G : J \rightarrow sSet$ in the quasi-category model structure. If $G(x) = \emptyset$, then $G'(x) = \emptyset$  for the projective cofibrant replacement $G'$ of $G$. Thus, we reduce to the case of ordinary colimits of simplicial sets, in which case the statement is obvious. 
\end{proof}

\begin{proof}[Proof of Theorem 4.2]
First we show that the map \cref{con5.5} (co)restricts to a functor $\mathbb{T}: Filt((C^{op})^{\Delta^{1}}, \mathcal{S}) \rightarrow Filt(C^{op}, \mathcal{S})^{\Delta^{1}}$. 

By the description of filtered functors $D \rightarrow \mathcal{S}$ from \cite[Section 5.3.5]{LurieHTT}, it suffices to show that $\mathbf{T}_{i}$ preserves the property of being a filtered colimit of representable functors. Because left Kan extensions preserve filtered colimits, we want to show that $\mathbf{T}_{i}$ takes representables to filtered colimits of representables. By the pointwise formula for Kan extensions (\cite[Proposition 6.4.9]{Cisinski}), we have
$$T_{i}(Hom_{C^{\Delta^{1}}}(-, y))(x) = colim_{C^{d^{i}}(z) = x} Hom_{C^{\Delta^{1}}}(x, z)) = Hom_{C^{\Delta^{1}}}(x \circ C^{s^{i}}, y)$$
since $x \circ C^{s^{i}}$ is the terminal object of $(C^{d^{i}})^{-1}(x) $.

We will produce a map  $\mathfrak{S} : Fun(C, \mathcal{S})^{\Delta^{1}} \rightarrow Fun(C^{\Delta^{1}}, \mathcal{S})$ which we will show corestricts to an inverse of $\mathbb{T}$. To do this we will closely follow the argument of \cite[Lemma IV.1.2]{Johnstone}, by producing the (co)units of the equivalences.

 Given $\alpha : U_{0} \rightarrow U_{1}$, we define $\mathfrak{S}(\alpha)$ by the pullback:

$$
\xymatrix
{
\mathfrak{S}(\alpha)(f) \ar[r]_{p_{f}^{2}} \ar[d]_{p_{f}^{1}} & U_{1}(codom(f)) \ar[d]^{T_{1}(f)} \\
U_{0}(dom(f)) \ar[r]_{\alpha_{dom(f)}} & U_{1}(dom(f))
}
$$ 

We want to show that $\mathfrak{S}$ corestricts to a functor 
$$
\mathbb{S} : Filt(C^{op}, \mathcal{S})^{\Delta^{1}} \rightarrow Filt((C^{op})^{\Delta^{1}},\mathcal{S})
$$

By the characterization of filtered functors from \cite[Section 5.3.5]{LurieHTT}, it suffices to show that $\mathfrak{S}$ takes transformations of filtered colimits of representables to filtered colimits of representables. Since every map of ind objects has a level representation, we reduce to showing that $\mathfrak{S}$ sends diagrams of representables to representables. 

Consider a natural transformation $\alpha: Hom_{C}(-, x) \xrightarrow{Hom_{C}(-, q)} Hom_{C}(-, y)$.
For a morphism $f: z \rightarrow w$ in $C^{\Delta^{1}}$, the description of mapping spaces in functor categories as ends from \cite{Haugseng} implies that we can describe $Map_{C^{\Delta^{1}}}(f, q)$ as the limit of
$$
\Delta_{/ \Delta^{1}} \rightarrow \Delta^{1} \times (\Delta^{1})^{op} \xrightarrow{(f, q^{op})} C \times C^{op} \xrightarrow{Map_{C}(-, -)} \mathcal{S}
$$
which, unwinding the definitions is exactly the pullback defining $\mathfrak{S}(\alpha)(f)$.

 Given $U \in Fun((C^{op})^{\Delta^{1}}, \mathcal{S})$, we have a natural transformation 
$$
 \Phi : U(-) \cong \int^{f} Hom_{C^{\Delta^{1}}}(-, f) \times U(f) \rightarrow  \int^{f} Hom_{C}(-, dom(f)) \times U(f) = \mathbf{T}_{0}(U)(-)
$$
By the proof of \cref{lem5.4}, this can be identified with a map $\Phi$ which takes $z_{1} \in U(f)$ to the element of the coend represented by $(id_{dom(f)}, z_{1})$. Similarly, we can construct a map $\Psi : U \rightarrow T_{1}(U)$  which takes $z_{1} \in U(f)$ to the element of the coend represented by $(id_{dom(f)}, z_{1})$.
We let the $\epsilon_{U} : U \rightarrow \mathfrak{S} \mathbf{T}U$ be the natural transformation given pointwise by the universal property of pullback in the commutative diagram:
$$
\xymatrix
{ 
U(f) \ar[drr]^{\Psi}  \ar[ddr]_{\Phi}  \ar@{.>}[dr]^>>>>>{\epsilon_{U}(f)} & & \\
& \ar[r] \ar[d] \mathfrak{S}\mathbf{T}(U)(f) & \ar[d]  \mathbf{T}_{1}(U)(codom(f)) \\
& \mathbf{T}_{0}(U)(dom(f))  \ar[r] & \mathbf{T}_{1}(U)(dom(f))
}
$$ 

We want to show that $\epsilon_{U}$ is an equivalence for $U \in Filt((C^{op})^{\Delta^{1}}, \mathcal{S})$. 
By the $\infty$-categorical Yoneda lemma, it suffices to show that 
$$
\pi_{0} \epsilon_{U}^{K}
$$
is an equivalence for each finite simplicial set $K$. 
Since we have equalities 
$$
\epsilon_{U}^{K} = \epsilon_{U^{K}}
$$
and $\pi_{0}U^{K}$ is filtered by \cref{lemxxx}, it suffices to prove that $\pi_{0}  \epsilon_{U}$ is an equivalence. Since $\epsilon'$ is constructed identically to the counit of the equivalence from \cite[Proposition B.4.1.2]{Johnstone}, we can repeat the argument there and use \cref{lemxxx}. to show that $\mathbf{\epsilon}'$ is a bijection.
\\

On the other hand, let us now produce natural isomorphisms $\eta_{i} : \mathbb{T}\mathbb{S}(U)_{i} \cong U_{i}$.

First, we will show that there is a natural equivalence $\mathbf{T}\mathfrak{S}(-)_{0} \rightarrow (-)_{0}$, which corestricts to an appropriate equivalence.

Let
$
p: Fun(\Delta^{1}, C^{op})_{/id_{C^{op}}} \rightarrow Fun(\Delta^{1},C)
$
be the projection map.
Then we can identify $\mathfrak{S}\mathbf{T}(U)_{0}(c)$ with the coend
$$
 \int^{f} Hom_{C}(c, dom(p(f))) \times \mathfrak{S}\mathbf{T}(U)(p(f))  \simeq \mathfrak{S}\mathbf{T}(U)(id_{c}) \times hom(c, c) \simeq  \mathfrak{S}(U)(id_{c}) = \mathbf{T}_{0}(U)(c)
$$
where the first two isomorphisms follow from the fact that $Fun(\Delta^{1}, C)_{/id_{c}},  hom(c, c)$ each have a terminal object ($id_c$) as well as \cref{lemxx}, and the final isomorphism is immediate from the definition of $\mathfrak{S}$.

 Let $\alpha : U_{0} \rightarrow U_{1}$ be a natural transformation of functors $C \rightarrow D$. 
We will define $\eta_{1} : \mathbb{T}\mathbb{S}(\alpha)_{1} \rightarrow (\alpha)_{1}$ be the formula:
\begin{multline*}
\int^{f} (id, \alpha_{dom(f)} p^{1}_{\alpha}) : \mathbb{T}\mathbb{S}(U)_{1}(c) = \int^{f \in C^{\Delta^{1}}} Hom_{C}(c, codom(f)) \times \mathbb{S}(\alpha)(f) \\
\longrightarrow  \int^{f \in C^{\Delta^{1}}} Hom_{C}(c, dom(f)) \times U_{1}(f)  \rightarrow U_{1}(c)
\end{multline*}

where last horizontal arrow is the equivalence of \cref{lem5.4}. Using \cref{lem5.4}, the map can be identified with the map given by 
\begin{equation}\label{helper}
((x, y) , g) \mapsto U_{1}(g)(y)
\end{equation} 

Since we have identifications 
$$
(\eta_1)_{X^{K}} = ((\eta_{1})_{X})^{K}
$$
we can reduce as before to showing that $
\pi_{0}  \eta_{1}$ is an pointwise bijection.  But since this unit component is constructed almost identically to the component of the unit of the equivalence from \cite[Proposition B.4.1.2]{Johnstone}, we can show $\pi_{0}\eta_{1}$ is an equivalence usng
  \cref{lemxxx} , \cref{helper}, as well as the argument from \cite[Proposition B.4.1.2]{Johnstone}.
  
  \end{proof}

\section{The $\infty$-Bicategory of Bounded Topoi}
\label{secboundedtopoi}
In this section, we are going to prove the $\infty$-categorical analogue of the main theorem of \cite[Section B.4.1]{Johnstone}.

\begin{theorem}\label{thm6.1}
Suppose that $\mathcal{B}$ is an $\infty-topos$. Then
$BTOP_{\mathcal{B}}$ has finite weighted colimits. 
\end{theorem}

\begin{lemma}\label{lem6.2}
Suppose that $f^{*} : \mathcal{B} \rightarrow \mathcal{C}$ is geometric morphism.  Then we have an equivalence $f^{*}(X^{K}) \simeq f^{*}(X)^{K}$, natural in $\mathcal{B}$-categories $X$ and finite $\infty$-categories $K$.
\end{lemma}

\begin{proof}

We note that by \cite[Remark 1.2.1.9]{Martini-Wolf-2}, we have an equivalence 
\begin{equation}\label{eq1}
C^{\Delta^{n}}_{0} \simeq C_{n}
\end{equation}
natural in $\Delta^{n}$. In particular, we have that 
$$
C^{K \times \Delta^{n}}_{0} \cong ((C^K)^{\Delta^{n}})_{0} \simeq C^{K}_{n}
$$
so it suffices to show that  $f^{*}(X^{K})_{0} \simeq (f^{*}(X)^{K})_{0}$. 

By \cite[Proposition 4.3.2.14]{LurieHTT}, we can write finite quasi-categories $K$ as a homotopy colimit of its nondegenerate simplices, so using \cref{eq1} and the fact that $f^{*}$ preserves finite limits
$$
f^{*}(X^{K}) \cong f^{*}(\underset{\Delta^{n} \rightarrow K}{\underset{\longleftarrow}{lim}} X^{\Delta^{n}} ) \cong \underset{\Delta^{n} \rightarrow K}{\underset{\longleftarrow}{lim}} f^{*}(X^{\Delta^{n}}) ) \cong  \underset{\Delta^{n} \rightarrow K}{\underset{\longleftarrow}{lim}} f^{*}(X)_{n} \cong  \underset{\Delta^{n} \rightarrow K}{\underset{\longleftarrow}{lim}} f^{*}(X)^{\Delta^{n}} \cong f^{*}(X)^{K}
$$

\end{proof}

\begin{example}\label{exam6.3}
In the case that $\mathcal{B} = Psh(C)$, where $C$ is an $\infty$-category, we want to identify the presheaf $\Omega_{\mathcal{B}}$. 
We have equivalences 
\begin{equation}\label{functorfunc}
Cat_{\mathcal{B}} \cong Sh_{Cat_{\infty}}(Psh(C)) \cong Fun(C^{op}, Cat_{\infty}) = Psh_{Cat_{\infty}}(C)
\end{equation}
whee the second equivalence comes from the universal property of $Psh(C)$ as the free cocompletion. Under this equivalence, we see that $\Omega_{\mathcal{B}}$
can be identified with the presheaf $c \mapsto Psh(C)_{/y_{c}} \cong \mathcal{S}$. That is, it is the constant presheaf with value $\mathcal{S}$.

On the other hand suppose that $\mathcal{B}$ is an $\infty$-topos, so $\mathcal{B}$ is a full subcategory of $Psh(C)$. Thus, we can identify $\mathcal{B}(A) \subseteq Psh(C)(A)$ is an inclusion of a full subcategory, and under the identification of \cref{functorfunc}, we see that $\mathcal{B}$ can be identified with a functor $F : \mathcal{C} \rightarrow Cat_{\infty}$ such that $F(c)$ is a full subcategory of $\mathcal{S}$ for each $c \in C$.

 \end{example}

\begin{construction}\label{con6.4}
Suppose that $\mathcal{B}$ is a left exact localization $L$ of $Psh(C)$. Then we have a map natural in $X, Y \in Cat(\mathcal{B})$
\begin{multline*}
Fun_{\mathcal{B}}(X^{\Delta^{1}}, Y) \simeq \int_{c} Fun(X(L(c))^{\Delta^{1}}, Y(L(c))) \rightarrow \int_{c} Fun(X(L(c)), Y(L(c)))^{\Delta^{1}} \simeq \\  \int_{c} Fun(X(L(c)), Y(L(c))^{\Delta^{1}} ) 
\simeq Fun_{\mathcal{B}}(X, Y^{\Delta^{1}}) \simeq Fun_{\mathcal{B}}(X, Y)^{\Delta^{1}}
\end{multline*}
Here the first and fourth map are the equivalences from  \cite[Lemma 2.4.2.4]{Martini-Wolf-2} (for the fourth note that in $Sh_{Cat_{\infty}}(\mathcal{B})$ the cotensoring $(-)^{\Delta^{1}}$ is given by applying $(-)^{\Delta^{1}}$ sectionwise. The second map is induced by the natural map of \cref{con5.5}. The remaining maps come from the universal property of the cotensoring. 

Note that this map is natural  with respect to maps $Cat(\mathcal{B}) \leftrightarrows Cat(\mathcal{A})$ induced by geometric morphisms of $\infty$-topoi.
\end{construction}

\begin{definition}\label{def6.5}
Let $Fin_{\mathcal{B}}$ be the collection of 
$\mathcal{B}$-categories which can be identified with locally constant functors of finite $\infty$-categories on $\mathcal{B}$. A  $\mathcal{B}$-category $J$ is said to be filtered if
the colimit functor  $colim :\underline{Fun}_{\mathcal{B}}(I, \Omega) \rightarrow \Omega$ is continuous for each $I \in Fin_{\mathcal{B}}$. 

We say that a functor $f : C \rightarrow \Omega_{\mathcal{B}}$ is filtered if it is contained in the free $Fin_{\mathcal{B}}$ cocompletion of $C$.

We will write $\underline{Filt}_{\mathcal{B}}(X)$ for the full subcategory of $\underline{Fun}_{\mathcal{B}}(X, \Omega_{\mathcal{B}})$ spanned by filtered functors, and $Filt_{\mathcal{B}}(X)$ for the global sections. Write $\underline{Fun}^{cc}_{\mathcal{B}}(X ,Y)$ for the full subcategory of $\underline{Fun}_{\mathcal{B}}(X, Y)$ spanned by cocontinuous functors and $Fun^{cc}_{\mathcal{B}}(X, Y)$ for its global sections. 

 \end{definition}

\begin{lemma}\label{lem6.6}

The functor in \cref{con6.4} corestricts to a functor 
$$
Fun_{\mathcal{B}}^{cc}(X^{\Delta^{1}}, Y) \rightarrow Fun_{\mathcal{B}}^{cc}(X, Y)^{\Delta^{1}}
$$
and 
$$
Filt_{\mathcal{B}}(X^{\Delta^{1}}) \rightarrow Filt_{\mathcal{B}}(X)^{\Delta^{1}}.
$$

\end{lemma}

\begin{proof}
To see this, note that since $\mathcal{B}$ is an $\infty$-topos, it can be identified with a full subcategory of $Psh(C)$ for some $\infty$-category $C$. Thus, we can assume that $\mathcal{B} = Psh(C)$. In this case, colimits are computed pointwise, so the result follows from the description of \cref{con6.4}. 

For the second statement, note that we can assume that $\mathcal{B} = Psh(C)$. In this case, using \cref{exam6.3} we see that the free $Fin_{\mathcal{B}} $-cocompletion functor is nothing more than applying free finite cocompletion sectionwise.  Thus, the result follows from the description of the map from \cref{con6.4}.

\end{proof}

\begin{lemma}\label{lem6.7}
For any $\infty$-topos $\mathcal{B}$, the free cocompletion functor  $Psh_{\mathcal{B}}(-) : CAT_{\mathcal{B}} \rightarrow TOP_{\mathcal{B}}$  takes cotensors with $\Delta^{1}$ to tensors with $\Delta^{1}$.
\end{lemma}

\begin{proof}

Suppose that $\mathcal{X}$ is a $\mathcal{B}$-topos with associated algebraic morphism $f_{*} : \mathcal{X} \rightarrow \mathcal{B}$.

We have a diagram
\begin{equation}\label{func}
\xymatrix
{
 Fun_{\mathcal{B}}^{cc}(\underline{Psh}_{\mathcal{B}}(C), \mathcal{X})^{\Delta^{1}}  \ar[rr]_>>>>>>>>>>>{ ((h_{C})^{*} )^{\Delta^{1}}} && Fun_{\mathcal{B}}(C, \mathcal{X}) )^{\Delta^{1}} \ar[r] & Fun_{\mathcal{X}}(f^{*}C, \Omega_{\mathcal{X}})^{\Delta^{1}} \ar[r] &  Fun_{\mathcal{X}}(f^{*}C, \Omega_{\mathcal{X}})^{\Delta^{1}}   \\
Fun_{\mathcal{B}}^{cc}(\underline{Psh}_{\mathcal{B}}(C^{\Delta^{1}}), \mathcal{X})  \ar[rr]_>>>>>>>>>>>{ ((h_{C})^{*} )} && Fun_{\mathcal{B}}(C^{\Delta^{1}}, \mathcal{X}) \ar[u] \ar[r] & Fun_{\mathcal{X}}(f^{*}(C^{\Delta^{1}}), \Omega_{\mathcal{X}})  \ar[u]   \ar[r] & \ar[u]  Fun_{\mathcal{X}}(f^{*}(C)^{\Delta^{1}}), \Omega_{\mathcal{X}})
 }
\end{equation}
where $h_{C}$ is the internal version of the Yoneda embedding from \cite{Martini-Wolf-2}, 
and  the vertical maps come from the natural map of \cref{con6.4}. The diagram commutes by the naturality of the construction of \cref{con6.4}, \cref{lem6.2} as well as the identification $f_{*}\Omega_{\mathcal{X}} = \mathcal{X}$. 

By Diaconescu's theorem (\cite[Proposition 3.8.1.2]{Martini-Wolf-2}), as well as \cite[Proposition 2.3.4.6 and 2.2.3.7]{Martini-Wolf-2} the horizontal maps (co)restrict to equivalences:
$$
 Fun_{\mathcal{B}}^{alg}(\underline{Psh}_{\mathcal{B}}(C), \mathcal{X})^{\Delta^{1}}  \simeq Filt_{\mathcal{X}}(f^{*}(C))^{\Delta^{1}}, 
 Fun_{\mathcal{B}}^{alg}(\underline{Psh}_{\mathcal{B}}(C^{\Delta^{1}}), \mathcal{X}) \simeq Filt_{\mathcal{X}}(f^{*}(C)^{\Delta^{1}})
$$

Thus, if we show that the rightmost vertical map of \cref{func} induces an equivalence
$$
Filt_{\mathcal{X}}(C^{\Delta^{1}}) \simeq Filt_{\mathcal{X}}(C)^{\Delta^{1}}
$$
we will get an induced equivalence
$$
Fun_{\mathcal{B}}^{alg}(\underline{Psh}_{\mathcal{B}}(C^{\Delta^{1}}), \mathcal{X})  \rightarrow Fun_{\mathcal{B}}^{alg}(\underline{Psh}_{\mathcal{B}}(C), \mathcal{X})^{\Delta^{1}} \
$$
which is natural in $\mathcal{X}$. 

In this case \cref{con6.4} reduces to proving that the natural map $Filt(C(c)^{\Delta^{1}}, \mathcal{X}(c) ) \rightarrow Filt(C(c), \mathcal{X}(c))^{\Delta^{1}}$ from \cref{con5.5} is an equivalence. Since $\mathcal{X}(c)$ is a full subcategory of $\mathcal{S}$ by \cref{exam6.3}, the result follows from \cref{thm5.2}.

\end{proof}

\begin{proof}[Proof of Theorem \cref{thm6.1}]
By the dual of \cref{cor3.7}, it suffices to show that $TOP_{\mathcal{B}}$ has a tensoring with $\Delta^{1}$ and its underlying $\infty$-category has finite colimits.
The second point,  is \cite[Proposition 3.2.6.6]{Martini-Wolf-2}. For the first point, note that if $\mathcal{X} = Psh_{\mathcal{B}}(C)$, then $\mathcal{X}\otimes \Delta^{1}$ exists by \cref{lem6.7}. In general, choose an accessible localization $ L : s Psh_{\mathcal{B}}(C) \leftrightarrows \mathcal{X} : i$. In this case, it is easy to see that if we form the pushout
$$
\xymatrix
{
Psh_{\mathcal{B}}(C) \times Psh_{\mathcal{B}}(C) \ar[r] \ar[d]_{(L \times L)} & \ar[d] Psh_{\mathcal{B}}(C) \otimes \Delta^{1} \\
 \mathcal{X} \times \mathcal{X} \ar[r] & P
}
$$
then $P$ satisfies the universal property of $\mathcal{X}\otimes \Delta^{1}$.
\end{proof}

\section{The Object Classifier} \label{theobjclass}

In this section, we will show the existence of the object classifier for $\infty$-topoi. 
\\

\begin{definition}\label{def7.1}
We will write $PR$ for the full subcategory of $CAT_{\infty}$ consisting of presentable $\infty$-categories.

We will write $(TOP^{alg})^{/ \mathcal{B}}$ for the  subcategory of $CAT_{\infty}^{/\mathcal{B}}$ spanned by $\mathcal{B}$-topoi and algebraic morphisms, and $(TOP^{geom})^{ \mathcal{B}/}$ for the  subcategory of $CAT_{\infty}^{\mathcal{B}/}$ spanned by $\mathcal{B}$-topoi and right adjoints of algebraic morphisms.
\end{definition}

\begin{remark}\label{rmkxxx}
It follows from \cite[Proposition 4.2.1.5]{LurieHTT} and \cref{con2.12} that there is an equivalence of $\infty$-categories 
$$
((TOP^{alg})^{/ \mathcal{B}})^{th} \simeq Top_{/ \mathcal{B}}^{alg}.
$$
\end{remark}

\begin{lemma}\label{lem7.2}
There is a functor $U: \text{TOP}_{\mathcal{B}} \to CAT_{\infty}$ of $\infty$-bicategories mapping a topos to its underlying category and a geometric morphism to its left adjoint. Moreover, this functor is corepresentable. 
\end{lemma}

\begin{proof}
Let $Fin_{\mathcal{B}} \subseteq Cat(\mathcal{B})$ be the full subcategory of finite $\mathcal{B}$-categories. 
Recall from \cite[Sections 1.2.15, 2.2.3]{Martini-Wolf-2} that there exists a free cocompletion functor $Psh^{Fin_{\mathcal{B}}}_{\mathcal{B}}(-) :CAT(\mathcal{B}) \rightarrow CAT(\mathcal{B})$, which satisfies the universal property 
$$
Filt_{\mathcal{B}}(Psh^{Fin_{\mathcal{B}}}_{\mathcal{B}}(\mathcal{C}), \mathcal{D}) \simeq Fun_{\mathcal{B}}(\mathcal{C}, \mathcal{D})
$$
for each $Fin$-cocomplete $\mathcal{B}$-category $\mathcal{D}$.  We have natural equivalences in topoi:  
\begin{multline*}
Hom_{TOP_{\mathcal{B}}}(Psh_{\mathcal{B}}(Psh^{Fin_{\mathcal{B}}}_{\mathcal{B}}(*)), \mathcal{E}) \simeq Filt_{\mathcal{B}}(Psh_{\mathcal{B}}^{Fin_{\mathcal{B}}}(*), \mathcal{E} ) \simeq Hom_{CAT(\mathcal{B})}(*, \mathcal{E}) \\
Hom_{CAT(\mathcal{S})}(*, \mathcal{E})  = Hom_{CAT_{\infty}}(*, \mathcal{E}) \simeq \mathcal{E}
\end{multline*}
where the first equivalences follows from \cite[Proposition 2.3.4.6 and 2.2.3.7]{Martini-Wolf-2}, as well as the special case of Diaconescu's theorem from \cite[Corollary 3.8.2.3]{Martini-Wolf-2},  the second is just the universal property of $Fin$-cocompletion and the third  is follows from the adjunction of \cref{con4.4}.
\end{proof}

\begin{lemma}\label{lem7.3}  
There is an equivalence of $\infty$-bicategories
$$
\Gamma : TOP_{\mathcal{B}} \rightarrow (TOP^{alg})^{/ \mathcal{B}}
$$
\end{lemma}
\begin{proof}

By the first paragraphs of \cite[Section 2.3.4]{ Martini-Wolf-2}, $Top_{\mathcal{B}}$ has a cotensoring over $Cat_{\infty}$. Thus, using \cref{lem2.11}, we have a chain of equivalences for two $\mathcal{B}$-topoi $A, B$ and a finite $\infty$-category $K$ 
\begin{multline*}
\mathbf{i}Hom_{(TOP^{alg})^{/ \mathcal{B}}}(\Gamma(A), \Gamma(B^{K}) ) \simeq \mathbf{i}Hom_{TOP_{\mathcal{B}}}(A, B^{K}) \simeq \mathbf{i}Fun(K, \Gamma\underline{Hom}_{TOP_{\mathcal{B}}}(A, B)) \\ \simeq \mathbf{i}Fun(K, \underline{Hom}_{(TOP^{alg})^{/ \mathcal{B}}}(\Gamma(A), \Gamma(B))).
\end{multline*}
 The final equivalence above comes from the fact that $\Gamma$ preserves internal hom in $Top_{\mathcal{B}}$ since it is an equivalence.
We conclude that the functor $\Gamma$ from \cref{thm4.11} preserves the cotensoring. 

That is, there is a commutative diagram 
$$
\xymatrix
{
\mathbf{i}Fun(K, Hom_{TOP_{\mathcal{B}}}(A, B)) \ar[r] \ar[d]_{\simeq} & \mathbf{i}Fun(K, Hom_{(TOP^{alg})^{/ \mathcal{B}}}(\Gamma(A), \Gamma(B)) \ar[d]^{\simeq} \\
Hom_{Top_{\mathcal{B}}}(A, B^{K})) \ar[r]_{\Gamma} &Hom_{(Top^{alg})^{/ \mathcal{B}}}(\Gamma(A), \Gamma(B^{K}))
}
$$
for each finite $\infty$-category $K$, where the vertical maps exhibit the universal property of the cotensor product. The top horizontal map is thus an equivalence. We conclude using
 \cite[Proposition 2.2.5.7]{LurieHTT}, that 
 $$
 Hom_{TOP_{\mathcal{B}}}(A, B))  \rightarrow Hom_{(TOP^{alg})^{/ \mathcal{B}}}(\Gamma(A), \Gamma(B))
 $$
 is an equivalence of $\infty$-categories. The fact that $\Gamma$ is an equivalence follows using the description of equivalences of $\infty$-bicategories from \cite[Remark 1.2.14]{Gagna-Harpaz-Lanari}. 

\end{proof}

\begin{lemma}\label{lem7.4}
There is an equivalence of $\infty$-bicategories 
$$
(TOP^{alg})^{/ \mathcal{B}} \simeq ((TOP^{geom})^{ \mathcal{B}/})^{op}
$$
\end{lemma}

\begin{proof}
Recall from \cite{Haugseng-2} that there is an $\infty$-bicategory  $ADJ(CAT_{\infty})$ of adjunctions of $\infty$-categories. Write $Fun(\Delta^{1}_{\flat}, CAT_{\infty})_{ladj}, Fun(\Delta^{1}_{\flat}, CAT_{\infty})_{radj}$ for the full subcategories of $Fun(\Delta^{1}, CAT_{\infty})$ spanned by left and right adjoint functors. 
Consider the map
$$
\Phi: TOP_{alg}^{/\mathcal{B}} \rightarrow Fun(\Delta^{1}_{\flat}, CAT_{\infty})_{ladj} \simeq ADJ(CAT_{\infty}) \simeq (Fun(\Delta^{1}_{\flat}, CAT_{\infty})_{radj})^{op}
$$
where the first map is induced by the inclusion of \cref{con2.12} and second (third) map are the equivalences sending an adjunction to its left (right) adjoint given by \cite[Theorem 4.4]{Haugseng-2}.
By construction, the arrows of $(TOP^{alg})^{/ \mathcal{B}}$ can be identified with squares 
$$
\xymatrix
{
\mathcal{E} \ar[r]_{h_{!}  g_{!}} \ar[d]_{g_{!}}& \mathcal{B} \ar[d]_{id} \\
\mathcal{F} \ar[r]_{h_{!}} & \mathcal{B}
}
$$
where each $ g^{!}, h^{!}$ are algebraic morphisms. By construction the map $\Phi$ takes this square to a diagram of the form: 
$$
\xymatrix
{
\mathcal{E} & \mathcal{B} \ar[l]_{g^{!} h^{!}  } \\
\mathcal{F}  \ar[u]_{g^{!}}  & \ar[l]_{h^{!}} \mathcal{B}  \ar[u]_{id}
}
$$
where each $g^{!}, h^{!}$ are the right adjoints of $g_{!}, h_{!}$.

Thus, the map $\Phi$ corestricts to an equivalence of the required kind. 
\end{proof}

\begin{theorem}\label{thm7.5}
There is a functor $U: \text{TOP}_{\mathcal{B}} \to PR$ which maps a topos to its underlying category and a geometric morphism to its right adjoint. This functor is represented by an object we denote $\mathcal{B}[\mathbb{O}]$.  
\end{theorem}

Consider the commutative diagram
$$
\xymatrix
{
TOP_{\mathcal{B}} \ar[r] \ar[d]_{\Gamma \simeq} & CAT(\mathcal{B}) \ar[d]^{\Gamma} \ar[r] & CAT(\mathcal{S}) \ar[d]^{ id} \\
(TOP^{alg})^{/ \mathcal{B}} \ar[r] & (CAT_{\infty})_{/\mathcal{B}}
 \ar[r] & CAT_{\infty} }
$$
where $\Gamma$ is the global sections functor and the horizontal composites are forgetful functors. The left vertical map is an equivalence by \cref{lem7.3}. Thus, the bottom horizontal map is corepresentable.

We note that the equivalences given in \cite[Theorem 4.4]{Haugseng-2} (co)restrict to yield an equivalence 
$$
\Psi : PR_{\infty}^{ladj} \simeq PR_{\infty}^{radj}
$$
between the subcategories of $PR$ spanned by left (right adjoints). The forgetful functor $F: (TOP^{alg})^{/ \mathcal{B}}\rightarrow CAT_{\infty}$ factors as
$$
(TOP^{alg})^{/ \mathcal{B}} \xrightarrow{F'} PR^{ladj} \xrightarrow{i_{1}} CAT_{\infty}
$$
where $i_{1}$ is the inclusion and $F'$ is a forgetful functor. Thus by the proof of \cref{lem7.2}, we conclude that $F'$ is corepresentable. There is a  commutative diagram 
$$
\xymatrix
{
(TOP^{alg})^{/ \mathcal{B}} \ar[r]_{F'}  \ar[d]_{\Phi} & \ar[d] PR^{ladj} \ar[d]^{\Psi} \\
((TOP^{geom})^{ \mathcal{B}/})^{op} \ar[r]_{(G')^{op}} & (PR^{radj})^{op}
}
$$
where $G'$ is a forgetful functor and $\Phi$ is the equivalence of \cref{lem7.4}. This diagram and the fact that $F'$ is corepresentable implies that the functor $(G')^{op}: (TOP^{geom})^{ \mathcal{B}/} \rightarrow PR$ is as well.

\section{Classifying topoi}

\label{classifying}
This section is devoted to the construction of various classifying $\infty$-topoi. As we have already discussed, this paper does not really take a logical stance on the topic, and will not have notion of geometric theory supported by a proof calculus. Insteat, we take a much more \textit{hands-on} definition of theory, which for the working mathematician could be understood as \textit{a construction}.

\begin{definition}
    A \textit{theory} is a functor $\mathbb{T}: TOP_{\mathcal{B}}^{op} \to CAT_{\infty}$. We say that a theory admits a classifying topos, or is classified by a topos if the functor is representable, i.e. there exists a topos $\mathcal{B}[\mathbb{T}]$ and an equivalence \[\mathbb{T}(-) \simeq  TOP_{\mathcal{B}}(-,\mathcal{B}[\mathbb{T}])\]
\end{definition}

The definition above is intentionally too general and has a very simple paradigm in mind: we identify a theory with its \textit{semantic trace}.

This notion of \textit{theory} and classifying topos is free-styled in \cite[Section 5]{anel2022left}, where the authors present some of our examples below. Our discussion will upgrade their understanding by relating it to the notion of \textit{geometric sketch} in the next subsection. Our treatment is indeed much less hand-wavy in assessing its expressive power and identifies clearly the range of expression of classifying topoi.

To explain our intuition, and present the first interesting example of theory and classifying topos, let us restrict to the special case in which $\mathcal{B}$ is the topos of spaces $\mathcal{S}$. Through the whole section, we will make this restriction several times as this is simultaneously the most interesting case, and the case in which it is easier and most paradigmatic to present the examples.

\begin{example}[Lawvere theories]
Let $\mathsf{L}$ be a Lawvere theory in the sense of \cite{berman2020higher}, then we have a canonically associated theory \[\mathbb{T}^\mathsf{L}:  TOP^{op}_{\mathcal{S}} \to CAT_{\infty}\] mapping a topos $\mathcal{E}$ to its $\infty$-category of $\mathsf{L}$-models in $\mathcal{E}$. Let $\widehat{\mathsf{L}}$ be the free completion of $\mathsf{L}$ under finite limits preserving the existing products \footnote{This can be identified with $(Psh_{\mathcal{S}}^{Fin_{\mathcal{S}}}(\mathsf{L}^{op}))^{op} $, where $Psh_{\mathcal{S}}^{Fin_{\mathcal{S}}}(-)$ is the free cocompletion with respect to finite colimits whose existence follows from \cite[Proposition 2.3.4.6 and 2.2.3.7]{Martini-Wolf-2}}, then it follows from Diaconescu theorem that the topos $Psh_{\mathcal{S}}(\widehat{\mathsf{L}})$ classifies the theory $\mathbb{T}^\mathsf{L}$ (see the argument of \cref{lem7.3}). 
\end{example}

\begin{remark}
As we have previously hinted, the prestack of models of a Lawvere theory $\mathbb{T}^\mathsf{L}$ can be understood as a \textit{construction} that we can perform on all topoi, and the existence of a classifier offers a compact object that packages that construction, allowing to study only the topos to infer properties of the whole construction.
\end{remark}

\begin{remark}
For example, one could be interested in understanding the \textit{endomorphisms} existing over the prestack of a Lawvere theory, i.e. the natural transformations of the form $\mathbb{T}^\mathsf{L} \to \mathbb{T}^\mathsf{L}$. Then it follows by Yoneda that those a classified by the endomorphisms of its classifying topos. In particular, they form an accessible $\infty$-category  with directed colimits (which was everything but trivial in first place!).
\end{remark}

\begin{remark}
    The theory of classifying topoi provides us with a new perspective on a topos, imported from usual $1$-dimensional topos theory. When a topos is \textit{on the right} of a geometric morphism, it is a \textit{placeholder} for a theory (in the case above, a Lawvere theory). When it is on the \textit{left} of a geometric morphism, it is a mathematical universe in which we are taking models. The geometric morphism itself \textit{is} the model (\textit{of} the theory on the right \textit{in} the universe on the left).
\end{remark}

Before moving to somewhat sophisticated and logical examples, we shall see other examples of theories and comment on them.

\begin{example}[The theory of objects]
        Among the theories we find the -- already discussed in the previous sections -- theory of objects, being the theory mapping an $\infty$-topos $\mathcal{E}$ to its underlying $\infty$-category, $\mathcal{E} \mapsto \mathcal{E}$, and mapping a geometric morphism to its inverse image, \[\mathbb{O}: TOP_{\mathcal{B}}^{op} \to CAT_{\infty}\]
    Now, it goes without saying that such theory is classified by the topos $\mathcal{B}[\mathbb{O}]$, as we have shown in the previous section. 
\end{example}

\begin{example}[$n$-truncated objects]
Similarly to the previous example, we can consider the mapping $\mathcal{E} \mapsto \mathcal{E}_{\le n}$, where $\mathcal{E}_{\le n}$ is the full subcategory of $\mathcal{E}$ consisting of $n$-truncated objects (see \cite[Section 6.5.1]{LurieHTT}). It is well known that inverse images preserve $n$-truncated objects as the truncation functors are left adjoints (\cite[Proposition 5.5.6.16]{LurieHTT})) and they commute with inverse images (\cite[Proposition 5.5.6.28]{LurieHTT})). It follows that we can define a theory \[\mathbb{O}_{\le n}: TOP_{\mathcal{S}}^{op} \to CAT_{\infty}.\]

This theory has a classifying topos $ \mathcal{S}[\mathbb{O}]_{\le n}$, which is moreover, an $n$-topos (see \cite[Theorem 6.4.1.5]{LurieHTT}).

\end{example}

\begin{example}[The theory of pointed objects]

For each pointed $\infty$-category $C$ (i.e. an $\infty$-category with a $0$ object), we will write $C_{\bullet}$ for the full subcategory of $C^{\Delta^{1}}$ consisting of objects $0 \rightarrow c$. Since geometric morphisms preserve zero objects, there is a functor $\mathcal{E} \rightarrow \mathcal{E}_{\bullet}$

\[\mathbb{O}_\bullet: TOP_{\mathcal{S}}^{op} \to CAT_{\infty}.\]

\noindent
By Diaconescu's theorem, this functor is represented by the $\infty$-topos $Psh_{\mathcal{S}}((Fin_{\mathcal{S}_{*}})^{op})$, where $Fin_{\mathcal{S}_{*}}$ is the $\infty$-category of finite pointed spaces (c.f. \cite[Notation 1.4.2.5]{lurie_higher_algebra}).

\end{example}

\begin{example}[The theory of spectra] \label{spectra}
Following \cite[Section 1.4]{lurie_higher_algebra}, for $\mathcal{K}$ a locally presentable infinity category we can construct $\mathsf{Sp}(\mathcal{K}_\bullet)$ the (locally presentable) category of spectra in $\mathcal{K}_\bullet$. By definition (\cite[Definition 1.4.2.8]{lurie_higher_algebra}), the spectrum objects can be identified with excisive functors $Fin_{\mathcal{S}_{*}} \rightarrow \mathcal{E}$, where $Fin_{\mathcal{S}_{*}}$ is the $\infty$-category of finite pointed spaces, and an excisive functor is a functor that takes pushouts to pullbacks.

Now, of course every $\infty$-topos is locally presentable, and thus we have the mapping $\mathcal{E} \mapsto \mathsf{Sp}(\mathcal{E}_\bullet)$.
For $f: \mathcal{E} \to \mathcal{F}$ a geometric morphism, because the inverse image preserves both finite limits and zero objects and thus the construction of spectra, providing us with a functor $f^\bullet: \mathsf{Sp}(\mathcal{F}_\bullet) \to \mathsf{Sp}(\mathcal{E}_\bullet)$. 
This association gives us the theory, \[\mathbb{S}: TOP_{\mathcal{S}}^{op} \to CAT_{\infty}.\]

The spectrum objects $\mathsf{Sp}(\mathcal{E}_\bullet)$ in the sense above are a presentation of the stabilization of the $\infty$-topos $\mathcal{E}$ in the sense of \cite{Harpaz-Nuiten-Prasma}. Thus, the spectrum functor above can be identified with other models of the stabilization of an $\infty$-category (see \cite[Remark 3.24]{Harpaz-Nuiten-Prasma}). Of particular importance is the classical case of symmetric $\Omega$-spectrum objects, which presents the stabilization by \cite[Proposition 4.15]{Robalo}. 

The symmetric $\Omega$-spectrum objects can be described as follows. Consider the subcategory $Sph \subseteq Fin_{\mathcal{S}_{*}}$ on the objects $S^{n}, n \ge 0$, such that
$$
Hom_{Sph}(S^{n}, S^{n+k}) = S^{k} \subseteq Hom_{Fin_{\mathcal{S}_{*}}}(S^{n}, S^{n+k})
$$
embedded via the adjoint to the natural homeomorphism $S^{n} \bigwedge S^{k} \simeq S^{n+k}$ and all other maps the constant map. We define $Sph^{\Sigma} \subseteq Fin_{\mathcal{S}_{*}}$ to be the subcategory obtained by adjoining to $Sph$ the morphisms generated by the action of the symmetric group $\Sigma_{n}$ on $S^{1} \wedge \cdots \wedge S^{1} \simeq S^{n} $. The $\Omega$-spectrum objects in $\mathcal{E}$ can be identified with objects of $Fun_{*}(Sph^{\Sigma}, \mathcal{E})$, for a pointed $\infty$-category $\mathcal{E}$. Thus, by Diaconescu's theorem, the classifying $\infty$-topos can be identified with $Psh_{\mathcal{S}}((Sph^{\Sigma})^{op})$.  

\end{example}

% In the discussion above we hope to have convinced the reader that even without leaving the reams of Lawvere theories, there is a quite broad families of natural constructions associated to topoi that we can describe, and many of them even have a geometric flavor. The table below summarises what we have discussed so far.

% \begin{table}[h!]
% \begin{tabular}{|l|l|}
% \hline
% a from $\mathcal{E}$ morphism into & correspond to              \\ \hline
%                                    & an object of $\mathcal{E}$ \\ \hline
%                                    &                            \\ \hline
%                                    &                            \\ \hline
% \end{tabular}
% \end{table}

\subsection{Geometric sketches} \label{subsecsketches}

We shall start with a short digression on the theory of \textit{sketches}. The community of category theory and categorical logic introduced sketches as compact presentations of theories. Their intention is to generalize the notion of Lawvere theory to a much broader context, taking the blueprint of \textit{functorial semantics} to its limit. There are many classical introductions to the notion of sketch in $1$-dimensional category theory, and the interested reader who wants to dive into the topic from a perspective closer to that of this paper can check \cite[2.F]{adamekrosicky94} and \cite[Vol II, 5.6]{borceux_1994}.

Sketches are an extremely broad framework that encompasses almost every possible (predicate-style) logical system  \cite[Chap. 5]{adamekrosicky94}. As an evidence of the latter statement, and more tightly related to the topic of this paper, classically it is shown that sketches can be used as an alternative -- purely categorical -- framework to encode geometric logic, via the notion of \textit{geometric sketch} (see \cite[Vol III, 4.3]{borceux_1994}.

In \cite{casacuberta2025sketchable}, the authors present a generalization of the classical theory of sketches to the context of $\infty$-categories. We refer to the paper for a \textit{crash course} on the topic, especially for a collection of examples. Notice that \cite{casacuberta2025sketchable} proves that every example that we have presented in the previous part of the section is indeed given by the category of models of a sketch, and thus this subsection subsumes the previous discussion.

   \begin{definition}
We say that a sketch $\Sigma$ in the sense of  \cite[Def 3.1]{casacuberta2025sketchable} is \textit{geometric} if its limit part consists of finite diagrams. No restrictions on its colimit part.    
   \end{definition} 

   \begin{example}[The prestack associated to a geometric sketch]
   Now, let  $\Sigma$ be a geometric sketch, then we have an associated prestack

   \[\textbf{Mod}(\Sigma,-): TOP_{\mathcal{S}}^{op} \to CAT_{\infty},\]
   mapping a topos $\mathcal{E}$ to the accessible category of models of $\Sigma$ in $\mathcal{E}$ in the sense of \cite[page 9]{casacuberta2025sketchable}. The construction is functorial because the sketch is geometric, indeed inverse images will preserve the colimit part because they are cocontinuous and the finite (!) limit diagram because they are lex.
   \end{example}

   \begin{remark}
   In the aforementioned paper the authors show that a vast number of constructions are sketchable, and it is easy to see that most of them (if not all) are given by geometric sketches.  Notice that geometric sketches generalize Lawvere theories on the spot, and in the classic case of $1$-dimensional topos theory this technology is equivalent to geometric logic.
 \end{remark}

 \begin{theorem}  \label{smallgeom}
   If a theory has a classifying topos, then it is the prestack of models of a (large) geometric sketch.
\end{theorem}
\begin{proof}
Let $\mathbb{T}$ be the theory and $\mathcal{B}[\mathbb{T}]$ be its classifying topos. We shall study only the case $\mathcal{S}[\mathbb{T}]$ for simplicity. Now, we equip $\mathcal{S}[\mathbb{T}]$ with a sketch structure, which we shall denote by $\Sigma[\mathbb{T}]$. In the colimit part we put all colimit cocone diagrams, and in the limit part we put all finite limit cone diagrams. It goes without saying that this is a geometric sketch. Now, it is clear that $\textbf{Mod}(\Sigma[\mathbb{T}],-) \simeq \textbf{Cocontlex}(\mathcal{S}[\mathbb{T}],-)$, where the latter is the category of cocontinuous and finite limit preserving functors. Observe that by the Adjoint Functor Theorem we have the following chain of equivalences, which finishes the proof.

\[\textbf{Mod}(\Sigma[\mathbb{T}],-) \simeq \textbf{Cocontlex}(\mathcal{S}[\mathbb{T}],-) \stackrel{\text{AFT}}{\simeq} TOP_{\mathcal{S}}(-,\mathcal{S}[\mathbb{T}]) \simeq \mathbb{T}(-).\]
\end{proof}

 \begin{theorem} \label{classtoposforgeoske}
   If a theory has a classifying topos, then it is the prestack of models of a small geometric sketch.
\end{theorem}
\begin{proof}
Let $\mathbb{T}$ be the theory, $\mathcal{B}[\mathbb{\mathbb{T}}]$ be its classifying topos and $\Sigma[\mathbb{T}]$ the large geometric sketch from the previous theorem. We shall provide a small geometric sketch $\Lambda[\mathbb{T}]$ and a morphism of sketches $i: \Lambda[\mathbb{T}] \to \Sigma[\mathbb{T}]$ that is a \textit{Morita equivalence}, in the sense that the induced morphism below induces a natural equivalence of prestacks. This will finish the proof.

\[\textbf{Mod}(\Sigma[\mathbb{T}],-) \simeq \textbf{Mod}(\Lambda[\mathbb{T}],-).\]

To do so, consider any site $(C,\Xi)$ of definition (whose underlying category is lex) for $\mathcal{B}[\mathbb{T}]$ in the sense of \cite[4.3.8]{anel2022left}. It follows from \cite[4.3.2 and the discussion right below]{anel2022left} that we can assume $\Xi$ to consist of maps $\delta: P \to yc$ where $P$ is a presheaf, and $yc$ is a representable. Now, we can turn $\delta$ into a cocone between the diagram induced by the category of elements $\pi_P: \text{Elts}(P) \to Psh(C)$ and the constant diagram $const_c: \text{Elts}(P) \to Psh(C)$. These will be our colimit cocones in $\Lambda[\mathbb{T}]$, while its limit cones will be all finite limiting cones. It is easy to see that the inclusion $i: \Lambda[\mathbb{T}] \to \Sigma[\mathbb{T}]$ is a morphism of sketches with these definitions, and that it has the required properties, again by (the proof of) \cite[4.3.10]{anel2022left}.
\end{proof}

\begin{remark}
    The two proofs above are largely influenced by \cite[Sections 4.1 and 4.2]{di2024sketches}. Here, we adapted the arguments in that paper to the technology and language of \cite{anel2022left}.
\end{remark}

\subsubsection{From geometric sketches to classifying topoi}
At this point we have learned that if a theory has a classifying topos then it was actually the prestack of models of a geometric sketch. In the next part of the subsection we will prove the converse implication, i.e. every geometric sketch admits a classifying topos.

The reader may have noticed that the previous examples have never mentioned the main result of the paper (\Cref{thm6.1}). We shall now prove that every geometric sketch has a classifying topos.

\begin{remark}
    In the discussion below, by \textit{inverter} (and later \textit{inserter} we inted the $\infty$-analogs of the weighted colimits discussed in \cite[Sec 4]{Kelly-Limits}. 
\end{remark}

\begin{theorem} \label{classtoposforgeoske}
    Let $\textbf{Mod}(\Sigma,-): TOP_{\mathcal{S}}^{op} \to CAT_{\infty}$ be the prestack associated to a geometric sketch, then such prestack has a classifying topos.
\end{theorem}
\begin{proof}
Let $\Sigma = (\mathcal{A}, \mathbb{L}, \mathbb{C})$ be a geometric sketch. We divide the proof in two steps.
\begin{itemize}
    \item[(Step 1)] When the sketch $\Sigma$ has no limit cones, nor colimit cocones, then it is easy to see that we have a classifier. This follows directly from Diaconescu, taking the free category with finite limits over $\mathcal{A}$.
    \item[(Step 2)] Now assume for simplicity that $\Sigma$ has precisely one cone. In practice such structure can be understood as a couple of diagrams $D, \mathsf{const}_x:I \rightrightarrows \mathcal{A}$, where one of the diagrams is constant on a certain object $c$, together with a natural transformation $\rho: D \Rightarrow \mathsf{const}_x$ between them. Now, we can identify both $\mathcal{A}$ and $I$ with their associated sketch of $A$ and $I$-diagrams, obtaining two sketches $\Sigma^\mathcal{A}$ and $\Sigma^I$. Then, we obtain the following diagram,

% https://q.uiver.app/#q=WzAsMyxbMCwwLCJcXHRleHRiZntNb2R9KFxcU2lnbWEsLSkiXSxbMSwwLCJcXHRleHRiZntNb2R9KFxcU2lnbWFee1xcbWF0aGNhbHtBfX0sLSkiXSxbMiwwLCJcXHRleHRiZntNb2R9KFxcU2lnbWFeSSwtKSJdLFsxLDIsIiIsMCx7Im9mZnNldCI6NX1dLFsxLDIsIiIsMix7Im9mZnNldCI6LTV9XSxbMCwxXSxbNCwzLCIiLDIseyJzaG9ydGVuIjp7InNvdXJjZSI6MjAsInRhcmdldCI6MjB9fV1d
\[\begin{tikzcd}[ampersand replacement=\&]
	{\textbf{Mod}(\Sigma,-)} \& {\textbf{Mod}(\Sigma^{\mathcal{A}},-)} \& {\textbf{Mod}(\Sigma^I,-)}
	\arrow[from=1-1, to=1-2]
	\arrow[""{name=0, anchor=center, inner sep=0}, shift right=5, from=1-2, to=1-3]
	\arrow[""{name=1, anchor=center, inner sep=0}, shift left=5, from=1-2, to=1-3]
	\arrow[ Rightarrow, from=1, to=0]
\end{tikzcd}\]

where $\textbf{Mod}(\Sigma,-)$ is pointwise (and thus globally) the inverter on the natural transformation. It follows by Yoneda that because both $\Sigma^\mathcal{A}$ and $\Sigma^I$ have a classifying topos, so does $\Sigma$. If there a set of cones, or cocones the argument is essentially the same.
\end{itemize}
\end{proof}

In the theory of sketches above we should understand the data of a cone (or a cocone) to be a kind of \textit{axiom} that we are enforcing on our theory. This is even more evident when we compute inverters, as we are taking those models which \textit{believe} a certain map to be an isomorphism.

To finish off our discussion we recall definitions taken  by \cite[B4.2]{Johnstone} which rest on this precise intuition.

\begin{definition}
    Let $\mathbb{T}$ be a theory.
    \begin{itemize}
        \item a \textit{geometric grounding} $(U,X)$ for a theory $\mathbb{T}$ is a couple where $X$ is a set and $U$ is natural transformation $U: \mathbb{T} \to \mathbb{O}^X$. $X$ is the \textit{arity} or \textit{sorting} of the grounding. The couple $(\mathbb{T},U)$ will be referred as \textit{geometric construct}, and we will always omit the specification of the sorting.
        \item Let $(\mathbb{T}, U_{\mathbb{T}}), (\mathbb{S}, U_{\mathbb{S}})$ geometric constructs (over the same sorting), a morphism of geometric constructs is a natural transformation as below:

        % https://q.uiver.app/#q=WzAsMyxbMCwwLCJcXG1hdGhiYntUfSJdLFsyLDAsIlxcbWF0aGJie1N9Il0sWzEsMSwiXFxtYXRoYmJ7T31eWCJdLFswLDEsImYiXSxbMCwyLCJVX3tcXG1hdGhiYntUfX0iLDJdLFsxLDIsIlVfXFxtYXRoYmJ7U30iXV0=
\[\begin{tikzcd}[ampersand replacement=\&]
	{\mathbb{T}} \&\& {\mathbb{S}} \\
	\& {\mathbb{O}^X}
	\arrow["f", from=1-1, to=1-3]
	\arrow["{U_{\mathbb{T}}}"', from=1-1, to=2-2]
	\arrow["{U_\mathbb{S}}", from=1-3, to=2-2]
\end{tikzcd}\]
One can imagine quite easily more flexible notions of morphisms of geometric constructs over different sortings, but we choose to keep the presentation simpler.
\item Let $U,V: \mathbb{T} \to \mathbb{O}^X$ be two geometric groundings of the same theory. We define $\mathbb{U}\Rightarrow\mathbb{V}$ to be the $\textit{inserter}$, in the diagram below and we shall refer to it as a \textit{functional extension} of $\mathbb{T}$.

% https://q.uiver.app/#q=WzAsMyxbMSwwLCJcXG1hdGhiYntUfSJdLFsyLDAsIlxcbWF0aGJie099XlgiXSxbMCwwLCJcXG1hdGhiYntVfSBcXFJpZ2h0YXJyb3cgXFxtYXRoYmJ7Vn0iXSxbMCwxLCJVIiwyLHsib2Zmc2V0IjoyfV0sWzAsMSwiViIsMCx7Im9mZnNldCI6LTJ9XSxbMiwwLCJqIiwwLHsic3R5bGUiOnsiYm9keSI6eyJuYW1lIjoiZGFzaGVkIn19fV1d
\[\begin{tikzcd}[ampersand replacement=\&]
	{\mathbb{U} \Rightarrow \mathbb{V}} \& {\mathbb{T}} \& {\mathbb{O}^X}
	\arrow["j", dashed, from=1-1, to=1-2]
	\arrow["U"', shift right=2, from=1-2, to=1-3]
	\arrow["V", shift left=2, from=1-2, to=1-3]
\end{tikzcd}\]

Recall that by definition an $\textit{inserter}$ is another theory $\mathbb{U}\Rightarrow\mathbb{V}$ equipped with a natural transformation $j: \mathbb{U} \Rightarrow \mathbb{V} \to \mathbb{T}$ and a natural transformation $\lambda: Vj \Rightarrow Uj$ with the expected universal property.
\item  Let $\mu: U \rightarrow V$ be a natural transformation between $U,V: \mathbb{T} \to \mathbb{O}^X$ be two geometric groundings of the same theory, we define $\mathbb{T}[\mu^{-1}]$ to be the \textit{inverter} of the diagram below,

% https://q.uiver.app/#q=WzAsMyxbMSwwLCJcXG1hdGhiYntUfSJdLFszLDAsIlxcbWF0aGJie099XlgiXSxbMCwwLCJcXG1hdGhiYntUfVtcXG11XnstMX1dIl0sWzAsMSwiVSIsMix7Im9mZnNldCI6NH1dLFswLDEsIlYiLDAseyJvZmZzZXQiOi01fV0sWzIsMCwiaiIsMCx7InN0eWxlIjp7ImJvZHkiOnsibmFtZSI6ImRhc2hlZCJ9fX1dLFs0LDMsIlxcbXUiLDAseyJzaG9ydGVuIjp7InNvdXJjZSI6MjAsInRhcmdldCI6MjB9fV1d
\[\begin{tikzcd}[ampersand replacement=\&]
	{\mathbb{T}[\mu^{-1}]} \& {\mathbb{T}} \&\& {\mathbb{O}^X}
	\arrow["j", dashed, from=1-1, to=1-2]
	\arrow[""{name=0, anchor=center, inner sep=0}, "U"', shift right=4, from=1-2, to=1-4]
	\arrow[""{name=1, anchor=center, inner sep=0}, "V", shift left=5, from=1-2, to=1-4]
	\arrow["\mu", Rightarrow, from=1, to=0]
\end{tikzcd}\]
    \end{itemize}
\end{definition}

We have already seen for the case of inverters how one can use these limit constructions to enforce behaviors over categories of models of a theory, and thus construct new classifying topoi.

\bibliographystyle{alpha}
\bibliography{classifyingBib}

\end{document}